\documentclass{ws-rmta}
\usepackage{xcolor}
\usepackage{tikz,pgfplots}
\pgfplotsset{
plotl/.style={blue,no marks,domain=-1:15,samples=50},
plotd/.style={red,no marks,ultra thick,domain=-2:2,samples=50}
}
\usetikzlibrary{positioning, shapes.arrows}

\renewcommand{\Re}{{\rm{Re}}}

\newcommand{\xb}{\mathbf{x}}
\newcommand{\ub}{\mathbf{u}}
\newcommand{\vb}{\mathbf{v}}
\newcommand{\Vb}{\mathbf{V}}

\newcommand{\ri}{\mathrm{i}}
\newcommand{\rI}{\mathrm{I}}
\newcommand{\CQ}{\mathcal{Q}}
\newcommand{\CTQ}{\widetilde{\mathcal{Q}}}
\newcommand{\CG}{\mathcal{G}}
\newcommand{\CTG}{\widetilde{\mathcal{G}}}
\newcommand{\im}{\operatorname{Im}}
\newcommand{\tsig}{\widetilde{\sigma}}

\newcommand{\norm}[1]{\left\lVert#1\right\rVert}
\newcommand{\T}{\mathrm{T}}

\numberwithin{equation}{section}
\newtheorem{thm}{Theorem}[section]
\newtheorem{exam}[thm]{Example}
\newtheorem{rem}[thm]{Remark}

\newtheorem{defn}[thm]{Definition}
\newtheorem{lem}[thm]{Lemma}
\newtheorem{assu}[thm]{Assumption}
\newtheorem{prop}[thm]{Proposition}
\newtheorem{con}[thm]{Condition}

\begin{document}


%
\catchline{}{}{}{}{}
%

\title{Spiked sample covariance matrices with possibly multiple bulk components 
}

\author{Xiucai Ding }

\address{Department of Mathematics, Duke University\\
Durham, NC 27710,  USA\\ 
\email{xiucai.ding@duke.edu} }

%
%
\maketitle
%

\begin{abstract}
In this paper, we study the convergent limits and rates of the eigenvalues and eigenvectors for  spiked sample covariance matrices whose  spectrum  can have multiple bulk components. Our model is an extension of Johnstone's spiked covariance matrix model.  Based on our results, we can extend many statistical applications  based on Johnstone's spiked covariance matrix model.  
\end{abstract}



\section{Introduction}
Sample covariance matrices play important roles in high dimensional data analysis, which find applications in many scientific endeavors. In the high dimensional regime, when the dimension is comparable to the sample size, the most popular and commonly used model is the spiked covariance matrix model proposed by Johnstone in \cite{IJ}.  Consider a sequence of $p$-dimensional i.i.d.  observations $\{\mathbf{x}_i\}_{i=1}^n$ satisfying $\mathbb{E} \xb_i=\mathbf{0}$ and $\text{Cov}(\xb_i)=\Sigma$ with the structure 
\begin{equation}\label{eq_sigmaoriginal}
\Sigma=\text{diag}\{\sigma_1, \cdots, \sigma_r, 1,\cdots,1\}, \ r>0 \ \text{is a fixed constant},
\end{equation} 
researchers are interested in extracting information about $\{\sigma_i\}_{i=1}^r$ using the sample covariance matrix $n^{-1} XX^{\mathrm{T}}, \ X=(\xb_i)_{i=1}^n.$

This model has been studied in the past decade in various papers, for instance see \cite{BBP,BS, BY, BKYY, DP}. For an orientation for such results, we refer to the recent review paper \cite{IP}.  Since the seminal work of \cite{BBP}, it is now understood that when $\sigma_i, i=1,2,\cdots,r $ are bounded and  above some critical values, the corresponding sample  eigenvalues of $n^{-1}XX^\T$ will converge to some deterministic values depending only on  $\{\sigma_i\}_{i=1}^r$ and the ratio $p/n.$ In the present paper, we {refer to} $\{\sigma_i\}_{i=1}^r$ as the {\em spiked population eigenvalues} and the rest as {\em bulk population eigenvalues}.  

One limitation of the assumption (\ref{eq_sigmaoriginal}) is that all the bulk population eigenvalues should be equal {unity}.  This is not realistic in many of the statistical applications. For instance, the samples $\{\xb_i\}_{i=1}^n$ may be a time series dataset and the entries of each sample $\xb_i$ undergo an AR process \cite{Yao}. This will make $\Sigma$ a Toeplitz matrix. {Furthermore, in the literature of signal processing \cite{YKN, YCND},  the population covariance matrix $\Sigma$ may have a
(known) finite number of distinct eigenvalues, and each of them with an unknown multiplicity. The multiplicity is also comparable to the sample size. In this case, there will be multiple clusters of eigenvalues, for instance see Figure 1 and Assumption 2.2 of \cite{YKN}.}
 In these situations, it will be more realistic to assume that the bulk eigenvalues of $\Sigma$ have a few clusters. In the present paper, we shall call such clusters as {\em bulk components} \cite{KY1}.

  Motivated by such applications, in the present paper, we generalize the spiked covariance matrix model (\ref{eq_sigmaoriginal}) by allowing the bulk eigenvalues to have a general density function and possibly multiple bulk components. In Figure \ref{figure_11}, we present such an example. 
  
  \begin{figure}[htb]
\centering
\includegraphics[height=5cm,width=10cm]{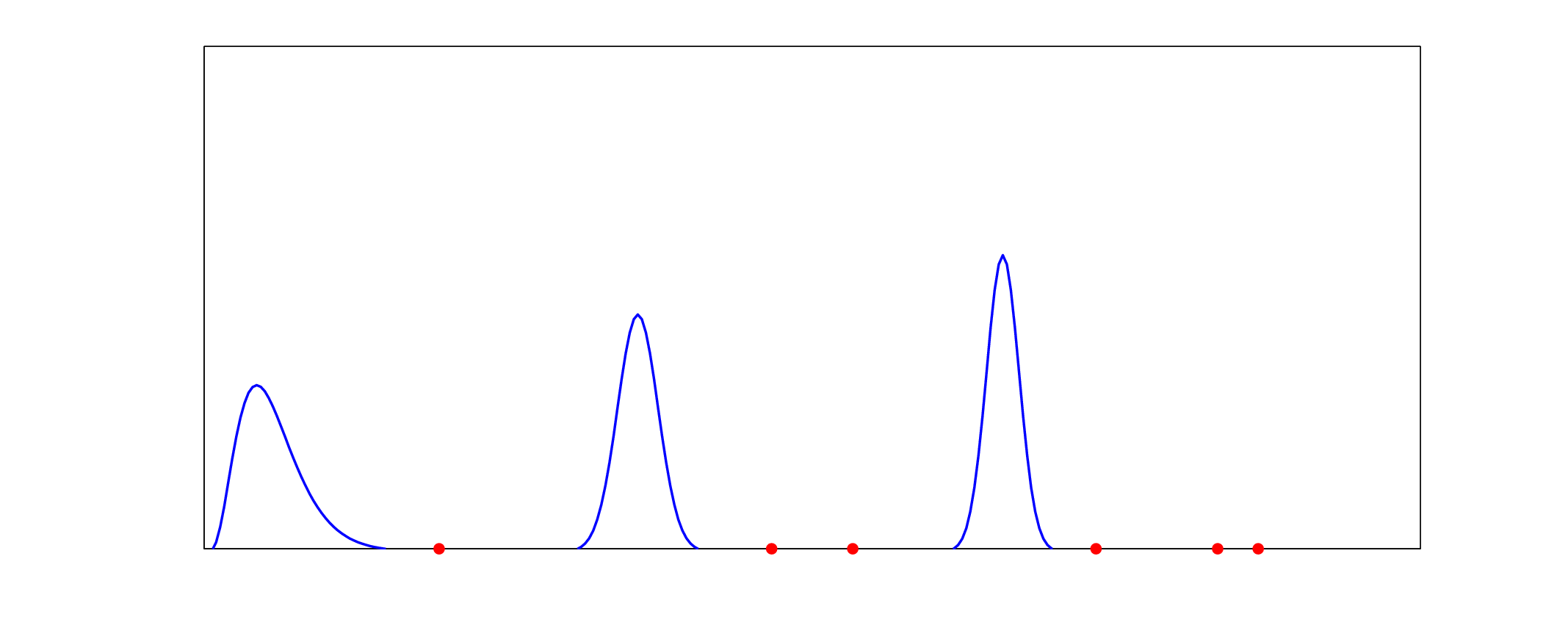}
\vspace*{-6mm}
\caption{The spectrum of the population covariance matrix contains three bulk components and there are three, two and one spike associated with the first, second and third bulk component respectively.} 
\label{figure_11}
\end{figure}


We mention that the non-spiked sample covariance matrices with possibly multiple bulk components have been studied in \cite{BPZ, XCD2, DY,  NKTW, HHN, KY1, LS}.  The null case is when the entries of the data matrix are i.i.d. {(i.e. $\Sigma=I$)}, where the {limiting} spectrum of the sample covariance matrices satisfies the celebrated Marchenko-Pastur (MP) law \cite{MP}. {Moreover, the local laws have been established in \cite{BEKYY,PY}.} {For general positive definite matrix $\Sigma,$ it is shown that the limiting spectrum of $n^{-1}XX^\T$ satisfies the deformed MP law \cite{SC}.  Recently, the local laws have been established in \cite{KY1}. } In this paper, we will extend the model from \cite{KY1} by adding a few spikes.  {It is notable that this idea has been proposed in \cite[Remark 3.8]{KY1} and we essentially implement the analysis here.} 

We point out that the spiked sample covariance matrix falls into the class of \emph{deformed random matrices,} which also include deformed Wigner matrix \cite{KY2, LS1, LS2} and deformed rectangular matrix \cite{BGGM, XCD1} as examples.
In the present paper, we study the convergent limits and rates of the eigenvalues {and} eigenvectors for a  new spiked covariance matrix model by allowing multiple bulk components (c.f. see the model definition in Section \ref{sec:model}). { 
In \cite{BKYY}, the authors have established analogous results when the bulk eigenvalues are equal unity. We basically extend the results of \cite{BKYY} to the case when $\Sigma$ is a general positive definite matrix satisfying the conditions of \cite{KY1}.  Our analysis relies on the methods and results of \cite{BKYY,KY1}. }

 Based on our results, we extend some statistical applications based on model (\ref{eq_sigmaoriginal}). We address two fundamental issues: the estimation of the numbers of bulk components and spikes and optimal nonlinear shrinkage of the eigenvalues. We believe that some other applications can be considered using our framework.

{Before concluding this section, we summarize the main contributions of our work: 

(i). We introduce a general spiked  covariance matrix model (c.f. (\ref{sigma_mostgene1})). This new model includes Jonstone's spiked covariance matrix model as a special example by allowing more general structure of the bulk population eigenvalues. Especially, the spectrum of our new model may have multiple bulk components with several spikes (c.f. Figure \ref{figure_11}).  This allows us to study more general spiked sample covariance matrices, for instance the spiked Toeplitz matrix in Example \ref{exam_top}. \\

(ii). For both supercritical (i.e. spikes have $O(1)$ separation from the bulk components)  and subcritical spikes (i.e. spikes have $o(1)$ separation from the bulk components), we obtain the first order limits of the corresponding
outliers and  the associated eigenvectors. Moreover, our results provide
a precise rate of convergence, which we believe to be optimal up to some $n^{\epsilon}$ factor {in the sense of stochastic domination (c.f. Definition \ref{defn_stochastic})}. \\

(iii). We prove large deviation bounds for the extremal non-outlier eigenvalues and eigenvectors. In particular, we prove that the extremal non-outlier eigenvalues will stick to the right-most edges of the associated bulk components. { Moreover, we provide the convergent rates of the extremal non-outlier eigenvectors near the edges of the bulk components. It turns out that the rates depend on the separation between the spikes and the bulk components. } \\

(iv). We extend some statistical applications based on Johnstone's spiked covariance matrix model. Especially, when the model has multiple bulk components, we provide an eigen-difference based statistic to estimate the number of bulk components and spikes. \\
}

%
%
%

Finally, to have a complete description of the principal components, we still need to consider the second order asymptotics, i.e. the limiting distribution of the outlier eigenvalues and eigenvectors. In the recent work \cite{BDWW}, the authors have obtained the joint distribution of the outlier eigenvalues and eigenvectors under Johnstone's spiked covariance matrix model.  We will generalize such results for our new model in the future work.

The present paper is organized as follows. In Section \ref{sec:model and mp}, we provide the notations, assumptions and definitions of the model. In Section \ref{sec:resultsandexamples}, we state our main results and provide some examples for explanation.  In Section
\ref{section_exapp}, we discuss some statistical applications. In Section \ref{section_tools}, we list the basic tools for our proofs. Finally, Sections \ref{section_value} and \ref{section_vector} are devoted to proving the results of the eigenvalues and eigenvectors respectively. 

%
%
%
%
%
%

\section{Deformed Marchenko-Pastur law and definition of the model} \label{sec:model and mp}
In this section, we explain the basic structure of the asymptotic eigenvalue density, define our model and list our key assumptions. We first introduce some notations. For a probability measure $\mu,$ we denote its \emph{Stieltjes transform} as 
\begin{equation}\label{eq_defnstiel}
m_{\mu}(z):=\int \frac{1}{x-z} \mu(dx), \ z \in \mathbb{C}^+,
\end{equation}
where $\mathbb{C}^+$ is the complex upper-half plane. For any $n \times n$ Hermitian matrix $H,$ the empirical spectral distribution (ESD) of $H$ is defined as 
\begin{equation*}
F_H^n(\lambda):=\frac{1}{n} \sum_{i=1}^n  \mathbf{1}_{\{\lambda_i(H) \leq \lambda\}}.
\end{equation*}
It is easy to see that the Stieltjes transform of the ESD of $H$ is given by 
\begin{equation*}
m_{H}(z)=\frac{1}{n} \operatorname{Tr} \mathcal{G}_H(z), \ z \in \mathbb{C}^+,
\end{equation*}
where $\mathcal{G}_H(z):=(H-z)^{-1}$ is the \emph{Green function} of $H.$

\subsection{Deformed Marchenko-Pastur law} We first consider the $p \times p$ matrix $\mathcal{Q}_1:=\Sigma^{1/2} XX^\T \Sigma^{1/2},$ where $\Sigma$ is a $p \times p$ positive definite deterministic matrix and $X$ is a random $p \times n$ matrix. {We} denote the dimensional ratio $c=\frac{n}{p}$ and assume that there exists a small constant $0<\tau \leq 1$ such that 
\begin{equation}\label{eq_regemic}
\tau \leq c \leq \tau^{-1}. 
\end{equation} 
We further suppose that the entries of $X=(x_{ij})$ are independent random variables such that 
\begin{equation} \label{eq_momentx12}
\mathbb{E} x_{ij}=0, \ \mathbb{E} x_{ij}^2=\frac{1}{n}. 
\end{equation}
In addition, we assume that there exists some large constant $C>0$ and for all $k \leq C,$ there exists some $C_k>0$ such that 
\begin{equation}\label{eq_mimentxhigh}
\mathbb{E}|\sqrt{n} x_{ij}|^k \leq C_k.
\end{equation}

Denote the eigenvalues of $\Sigma$ by
$\sigma_1 \geq \sigma_2 \geq \cdots \geq \sigma_p  >0,$ and the empirical spectral distribution (ESD) of $\Sigma$ by
\begin{equation} \label{def_pi}
\pi:=\frac{1}{p} \sum_{i=1}^p  \delta_{\sigma_i},
\end{equation}
where $\delta_{\cdot}$ is the Dirac Delta measure. We further assume that 
\begin{equation} \label{assm_11}
\sigma_1 \leq  \tau^{-1}, \  \pi([0,\tau]) \leq 1-\tau. 
\end{equation}

We know from \cite{KY1, MP, Silver} that the ESD of $\mathcal{Q}_1$ converges to a deterministic limit, which we shall call \emph{the deformed MP law.}  
The  deformed MP law can be best formulated using the Stieltjes transform. Next, we  follow \cite[Section 2.1]{KY1} to state the preliminary results of the deformed MP law.
We conclude from \cite[Lemma 2.2]{KY1} that if $\pi$ is compactly supported on $\mathbb{R},$ then for each $z \in \mathbb{C}^+,$ there exists a unique solution $m \equiv m(z) \in \mathbb{C}^+$ satisfying

\begin{equation} \label{selfconsistenteuqation1}
\frac{1}{m}=-z+\frac{1}{c} \int \frac{x}{1+m x}\pi (dx),
\end{equation} 
where $c=n/p$ and $\pi$ is defined in (\ref{def_pi}). 
{{Recall} the definition of \emph{asymptotic density}, which is introduced in \cite[Definition 2.3]{KY1}.
\begin{defn}[Asymptotic density] We define the deterministic function $m \equiv m_{\Sigma,n}: \mathbb{C}^+ \rightarrow \mathbb{C}^+$ as the unique solution of (\ref{selfconsistenteuqation1}) with $c=n/p$ and $\pi$ defined in (\ref{def_pi}). We denote by $\rho \equiv \rho_{\Sigma, n}$ the probability measure associated with $m,$ (i.e. $m$ is the Stieltjes transform of $\rho$), and call it the asymptotic eigenvalue density. 
\end{defn}
}

{We remark that the asymptotic density $\rho$ is indeed a probability measure and can possibly depend on $n$. } It is clear that both $m$ and $\rho$ are well-defined in our setting {since $\pi$ is discrete}. Moreover,  the behaviour of $\rho$ can be entirely understood by the analysis of the following function $f${
\begin{equation}\label{defnf}
z=f(m), \ \operatorname{Im} m > 0, \ \ \text{where}  \ f(x):=-\frac{1}{x}+\frac{1}{c} \sum_{i=1}^p \frac{p^{-1}}{x+\sigma_i^{-1}}.
\end{equation} 
}
We point out that 
\begin{equation} \label{key_quantity}
f(m(z))=m(f(z))=z.
\end{equation}
We refer the readers to  \cite[Section 5]{SC} and \cite{HHN,KY1} for more details. 
 We next summarize the properties of $f$ defined in (\ref{defnf}), it can be found in \cite[Lemmas 2.4, 2.5 and 2.6]{KY1}.  With the following properties, we can understand the behavior of $\rho.$ 
\begin{lem} \label{propertiesofrho} Denote $\overline{\mathbb{R}}=\mathbb{R} \cup \{\infty \}$, then $f$ defined in (\ref{defnf}) is smooth on the $p+1$ open intervals of  $ \overline{\mathbb{R}}$ defined through
\begin{equation*}\label{defn_I}
I_1:=(-\sigma_1^{-1},0), \ I_i:=(-\sigma_i^{-1}, -\sigma_{i-1}^{-1}), \ i=2,\cdots, p, \ I_0:= \overline{R} \backslash \cup_{i=1}^p \bar{I}_i.
\end{equation*}
We also introduce a multiset $\mathcal{C} \subset \overline{\mathbb{R}}$ containing the critical points of $f$, using the conventions that a nondegenerate critical point is counted once and a degenerate critical point will be counted twice. In the case $c=1,$  $\infty$ is a nondegenerate critical point. With the above notations, we have
\begin{itemize}
\item $\vert \mathcal{C} \cap I_0 \vert=\vert \mathcal{C} \cap I_1 \vert=1$ and $\vert \mathcal{C} \cap I_i \vert \in \{0,2\}$ for $i=2, \cdots, p.$ Therefore, $\vert \mathcal{C} \vert=2q$ is even, where for convenience, we denote by $x_1 \geq x_2 \geq \cdots \geq x_{2q-1}$ the $2q-1$ critical points in $I_1 \cup \cdots \cup I_p$ and {by} $x_{2q}$ the unique critical point in $I_0$.
\item Denote $a_k:=f(x_k) $, we have $a_1 \geq \cdots \geq a_{2q}.$ Moreover, we have $x_k=m(a_k)$ by assuming $m(0):= \infty$ for $c=1$. Furthermore, for $k=1,\cdots,2q,$ there exists a constant $C$ such that $0\leq a_{k} \leq C$. 
\item $\operatorname{supp} \ \rho \cap (0, \infty)=(\cup_{k=1}^q[a_{2k}, a_{2k-1}])\cap (0, \infty)$. 
\item  For some small constant  $\nu>0,$ when $x \in [x_{2k-1}-\nu, x_{2k-1}+\nu],\ {k=1,2,\cdots,q,}$
\begin{equation}\label{fderivative}
f^{\prime}(x)=O(|x-x_{2k-1}|), \ f(x)-a_{2k-1}=O(|x-x_{2k-1}|^2).
\end{equation}
\end{itemize}
\end{lem}
The above lemma shows that the asymptotic density $\rho$ can have $q \geq 1$ different bulk components with edges $\{a_k\}_{k=1}^{2q}.$ {For the reader's convenience,  we replicate Figure 2.1 of \cite{KY1} to illustrate the properties of the function $f$ and the density $\rho.$  }
%

  \begin{figure}[htb]
\centering\includegraphics[height=6cm,width=13cm]{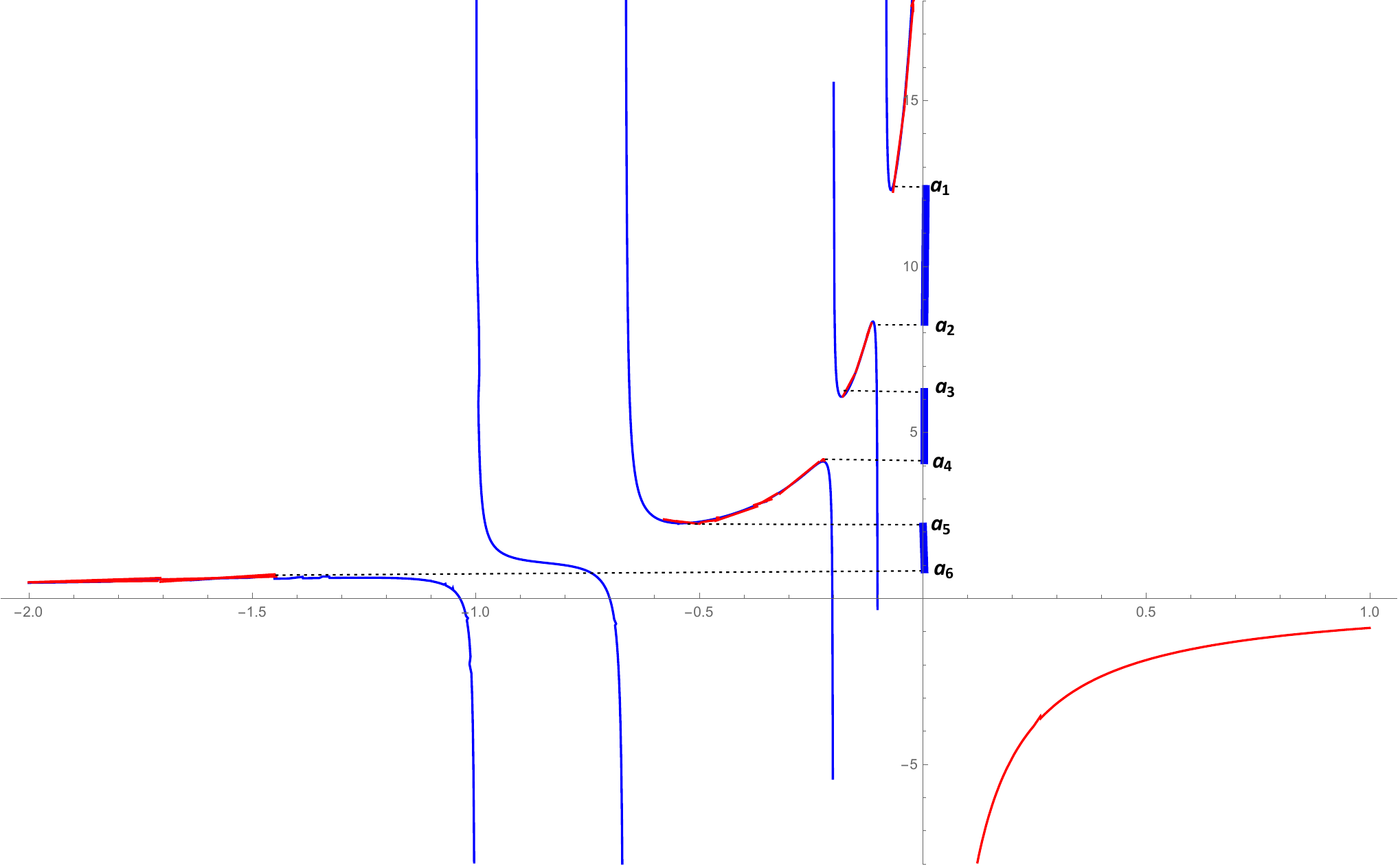}
\vspace*{-3mm}
\caption{The function $f(x)$ for $c^{-1} \pi=0.01 \delta_{10}+0.01\delta_5+0.05 \delta_{1.5}+0.03 \delta_1$. Here $c=10.$  It is clear that we have $q=3$ connected components. The support of $\rho$ is indicated with thick blue lines on the vertical axis. The inverse of $m|_{\mathbb{R} \backslash \operatorname{supp} \rho}$ is drawn in red.} 
\label{figure_1}
\end{figure}


\subsection{Definition of the model and assumptions}\label{sec:model} This subsection is devoted to defining our model and providing the necessary assumptions.  We first impose some  regularity conditions on $\Sigma,$ which are proposed in \cite[Definition 2.7]{KY1}. Roughly speaking, the regularity condition rules out the spikes from the spectrum of $\mathcal{Q}_1.$ 

\begin{assu} \label{assu_regularity} Fix  $\tau>0$,  we assume that \\
(i). The edges $a_k,\ k=1,\cdots, 2q$ are regular in the sense that 
\begin{equation} \label{defnregularityequation1}
a_k \geq \tau, \ \min_{l \neq k} \vert a_k-a_l \vert \geq \tau, \ \min_{i} \vert x_k+\sigma_{i}^{-1} \vert \geq \tau.
\end{equation}
(ii). The bulk components $k=1,\cdots,q$ are regular in the sense that for any fixed $\tau^{\prime}>0$ there exists a constant $\nu \equiv \nu_{\tau,\tau^{\prime}}$ such that the density of $\rho$ in $[a_{2k}+\tau^{\prime}, a_{2k-1}-\tau^{\prime}]$ is bounded from below by $\nu$.
\end{assu}

The second condition in (\ref{defnregularityequation1}) states that the gap in the spectrum of $\rho$ adjacent to $a_k$ can be well separated when $n$ is sufficiently large. The third condition ensures a square root behaviour of $\rho$ in a small neighborhood of $a_k$. 
The bulk regularity condition (ii) imposes a lower bound on the density of eigenvalues away from the edges such that we can study different bulk components separately.

We now propose our new model by adding a few spikes on each of the bulk components. Denote the spectral decomposition of $\Sigma$ as
\begin{equation*}
\Sigma=\sum_{i=1}^p \sigma_i \mathbf{v}_i \mathbf{v}_i^\T, \ \sigma_1 \geq \sigma_2 \geq \cdots \geq \sigma_p>0.
\end{equation*} 
Since there exist $q \geq 1$ bulk components,  we order them in the decreasing order according to the values of their right-most edges. We first relabel the eigenvalues $\{\sigma_i\}_{i=1}^p$ for each bulk component.  To the end, for $k=1, 2, \cdots, q,$ we define the \emph{classical number of eigenvalues in the $k$th bulk component} through  (for instance, see \cite[Lemma A.1]{KY1})
{
\begin{equation*}
n_k=\sum_{i=1}^p \mathbf{1}(x_{2k} \leq -\sigma_i^{-1} \leq x_{2k-1} ), \  k=1,2,\cdots, q.
\end{equation*}
}
It is easy to see that $\sum_{k=1}^q n_k=p. $
With the above notations, for $k=1,2,\cdots,q$ and $i=1,2,\cdots, n_k,$ we relabel the eigenvalues for $\Sigma$ by denoting {
\begin{equation} \label{defn_relabel1}
\sigma_{k,i}:=\sigma_{i+\sum_{l=1}^{k-1} n_l}. 
\end{equation}
}
Similarly, we can relabel the eigenvectors $\mathbf{v}_i$ as $\mathbf{v}_{k,i}.$ We illustrate the relabellings (\ref{defn_relabel1}) in Figure \ref{fig_relabel}. 

\begin{figure}[h]
\begin{tikzpicture}
\tikzset{box/.style={arrow box, draw=black}};
\node[box, arrow box arrows={south:1cm}] at (0,0){$\sigma_1 \geq \sigma_2 \geq \cdots \geq \sigma_{n_1}$};
\node at (0,-1.7){\fbox{$\sigma_{1,1} \geq \sigma_{1,2} \geq \cdots \geq \sigma_{1,n_1}$}} ;
\node at (2.1,0){$ \geq $};
\node at (2.4,-1.7){$ \geq $};
\node[box, arrow box arrows={south:1cm}] at (5,0){$\sigma_{n_1+1} \geq \sigma_{n_1+2} \geq \cdots \geq \sigma_{n_1+n_2}$};
\node at (5,-1.7){\fbox{$\sigma_{2,1} \geq \sigma_{2,2} \geq \cdots \geq \sigma_{2, n_2}$}};
\node at (8,0){$ \geq $};
\node at (8.7,0){$ \cdots \cdots$};
\node at (7.7,-1.7){$ \geq $};
\node at (8.4,-1.7){$ \cdots \cdots $};
\end{tikzpicture}
\caption{Relabellings of the eigenvalues of $\Sigma.$}\label{fig_relabel}
\end{figure}

Armed with the above preparation, we now construct our new model.  Given some fixed constant $r>0,$ we add $r $ spikes to the spectrum of $\Sigma$   and  suppose that there are $r_k, \ k=1,2,\cdots, q,$ spikes associated with the $k$th bulk component such that $\sum_{k=1}^q r_k =r. $  Let $\mathcal{I}$ be the index set containing the indices of the spikes, where
{
\begin{equation}\label{eq_index_location}
\mathcal{I}:=\{(k,i) \vert 1 \leq k \leq q, \ 1 \leq i \leq r_k\}.
\end{equation}
}
Then the new spiked covariance matrix model is defined as  
\begin{equation}\label{cov_gene}
\widetilde{\mathcal{Q}}_1:=\widetilde{\Sigma}^{1/2}XX^\T \widetilde{\Sigma}^{1/2},
\end{equation}
 where $\widetilde{\Sigma}$ is denoted by
 {
\begin{equation} \label{sigma_mostgene1}
\widetilde{\Sigma}=\sum_{i=1}^p \tilde{\sigma}_i \mathbf{v}_i \mathbf{v}_i^\T=\sum_{k=1}^q \sum_{i=1}^{n_k} \widetilde{\sigma}_{k,i} \mathbf{v}_{k,i} \mathbf{v}_{k,i}^\T, \ \text{where} \ \widetilde{\sigma}_{k,i}=\sigma_{k,i} , \ (k,i) \notin \mathcal{I}.  
\end{equation}
{Here we assume that $-\widetilde{\sigma}_{k,i}^{-1}<x_{2(k-1)}$ for $(k,i) \in \mathcal{I}.$} 
}



%



 Next, we introduce the set $\mathcal{O}$ through the relabellings (\ref{defn_relabel1}) with $\mathcal{O}=\bigcup_{k=1}^q \mathcal{O}_k,$ where {$\mathcal{O}_k, k=1,2,\cdots, q$ are disjoint index sets and $\mathcal{O}_k$} is defined as   
{
\begin{equation}\label{defn_oplus}
\mathcal{O}_k=\{i: x_{2k-1} + n^{-1/3}\leq-\tilde{\sigma}_{k,i}^{-1}<x_{2(k-1)}\}, \ k=1,2, \cdots, q,
\end{equation}
} 
where we use the convention that $x_0=+\infty.$ {Since $x_1 \geq x_2 \geq \cdots \geq x_{2q},$ by Lemma \ref{propertiesofrho}, we find that $\widetilde{\sigma}_{k,i}< \widetilde{\sigma}_{l,j}$ for $l<k$ and $ 1\leq i \leq n_k, 1 \leq j \leq n_l.$} 
For $k=1,2,\cdots,q,$ we denote $r_k^+=|\mathcal{O}_k|, k=1,2,\cdots,q.$ We will see later that each $(k,i) \in \mathcal{O}$ gives rise to an outlier eigenvalue near some location depending on $\Sigma, \widetilde{\sigma}_{k,i}$  and the ratio $c.$ We define the set $\mathcal{N}=\sum_{k=1}^q \mathcal{N}_k,$ where $\mathcal{N}_k$ contains the indices of $\widetilde{\sigma}_{k,i},$ which do not satisfy the condition in (\ref{defn_oplus}), i.e., for $(k,i) \in \mathcal{N}_k,$ 
\begin{equation*}
-\widetilde{\sigma}_{k,i}^{-1}< x_{2k-1}+n^{-1/3}, \ (k,i) \in \mathcal{I}.
\end{equation*}

To avoid repetition, we summarize the basic assumptions for future reference.
\begin{assu}\label{assum_main}  We assume that (\ref{eq_regemic}), (\ref{eq_momentx12}), (\ref{eq_mimentxhigh}), (\ref{assm_11}), (\ref{sigma_mostgene1}) and Assumption \ref{assu_regularity}  hold true. 
\end{assu}

\section{Main results and examples}\label{sec:resultsandexamples} \subsection{Main results} 
In this section, we state our main results. We first introduce the following definition. It is first introduced in \cite{BEKYY} and makes precise statement of the form " $\mathsf{X}$ is bounded with high probability by $\mathsf{Y}$
up to small powers of $n$". 

\begin{defn} [Stochastic domination]\label{defn_stochastic} Let
 \begin{equation*}
 \mathsf{X}=(\mathsf{X}^{(n)}(u):  n \in \mathbb{N}, \ u \in \mathsf{U}^{(n)}), \   \mathsf{Y}=(\mathsf{Y}^{(n)}(u):  n \in \mathbb{N}, \ u \in \mathsf{U}^{(n)}),
 \end{equation*}
be two families of nonnegative random variables, where $\mathsf{U}^{(n)}$ is a possibly $n$-dependent parameter set. We say that $\mathsf{X}$ is stochastically dominated by $\mathsf{Y},$ uniformly in $u,$ if for all small {$\epsilon>0$} and {large $ \varphi>0,$} we have 
\begin{equation*}
\sup_{u \in \mathsf{U}^{(n)}} \mathbb{P} \Big( \mathsf{X}^{(n)}(u)>n^{\epsilon}\mathsf{Y}^{(n)}(u) \Big) \leq n^{- \varphi},
\end{equation*}   
for large enough $n \geq  n_0(\epsilon, \varphi).$ In addition,  we use the notation $\mathsf{X}=O_{\prec}(\mathsf{Y})$ if $|\mathsf{X}|$ is stochastically dominated by $\mathsf{Y},$ uniformly in $u.$ Throughout this paper, the stochastic domination will always be uniform in
all parameters (mostly are matrix indices and the spectral parameter $z$) that are not explicitly fixed.  

Moreover, for any $n$-dependent event $\Xi,$ we say it is a high-probability event if $1-\mathbf{1}(\Xi) \prec 0.$
 \end{defn}

We now introduce some notations. Let $\mu_1 \geq \mu_2 \geq \cdots \geq \mu_{p \wedge n}$ be nontrivial eigenvalues of $\widetilde{\mathcal{Q}}_1$ and $\{\mathbf{u}_i\}$ be the associated eigenvectors. We relabel the eigenvalues and eigenvectors in the same way as (\ref{defn_relabel1}).  We first introduce the results for the eigenvalues. Recall $f$ in (\ref{defnf}). 

%
%
%
%
%
%
%
%
%
\begin{thm}\label{thm_evout}  Suppose Assumption \ref{assum_main} holds. For $k=1,2,\cdots,q$ and $i \in \mathcal{O}_k$ defined in (\ref{defn_oplus}),  we have  
\begin{equation}\label{outlier_out}
| \mu_{{k,i}}-f(-\widetilde{\sigma}_{k,i}^{-1}) | \prec n^{-1/2}(-\widetilde{\sigma}_{k,i}^{-1}-x_{2k-1})^{1/2}.
\end{equation}
Moreover, for any fixed large constant  $\varpi>0$ and  $k=1,2,\cdots, q, \ r^+_k+1 \leq i \leq \varpi,$  we have 
\begin{equation}\label{outlier_nonout}
|\mu_{{k,i}}-f(x_{2k-1})| \prec n^{-2/3}.
\end{equation}
\end{thm}
{
Theorem \ref{thm_evout} is an analogous result of \cite[Theorem 2.3]{BKYY}. The above theorem gives precise locations and rates of the outlier and the extremal non-outlier eigenvalues. For the outlier eigenvalues, they will locate around their \emph{classical locations} $f(-\widetilde{\sigma}_{k,i}^{-1})$ and for the extremal non-outlier eigenvalues, they will locate around the right-most edge of the bulk component. Moreover, the fluctuation of the outlier changes
from the order $n^{-1/2}(-\widetilde{\sigma}_{k,i}^{-1}-x_{2k-1})^{1/2}$ to $n^{-2/3}$ when $(-\widetilde{\sigma}_{k,i}^{-1}-x_{2k-1})$ crosses the scale { $n^{-1/3}.$ }  }

Next, we state the results of the outlier eigenvectors.  For $k=1,2,\cdots,q$ and $i \in \mathcal{O}_k$, denote {
\begin{equation}\label{defn_converlimit}
a_{k,i}:=\tilde{\sigma}_{k,i}^{-1} \frac{f^{\prime}(-\tilde{\sigma}_{k,i}^{-1})}{f(-\tilde{\sigma}_{k,i}^{-1})}.
\end{equation}
}
Further, we denote  
\begin{equation} \label{defn_nuoij}
\nu_{ij}^k:=
\begin{cases}
\min_{l \neq i}|-\tsig_{k,l}^{-1}+\tsig_{k,j}^{-1}|, & \ \text{if} \ i=j ; \\
|-\tsig_{k,i}^{-1}+\tsig_{k,j}^{-1}|, & \ \text{if} \ i \neq j. 
\end{cases} 
\end{equation}

\begin{thm}\label{thm_eveout} Suppose Assumption \ref{assum_main} holds.  For $k=1,2,\cdots,q, i,j \in \mathcal{O}_{k}, $ we have that 
\begin{equation}\label{eveout_eq}
\left| \langle \mathbf{u}_{{k, i}}, \mathbf{v}_{k,j} \rangle^2-\delta_{ij} a_{k,i} \right| \prec  \delta_{ij} n^{-1/2}(-\tsig_{k,i}^{-1}-x_{2k-1})^{-1/2} +n^{-1} (\nu_{ij}^k)^{-2}.
\end{equation}
Further, for $1 \leq k_1 \neq k_2 \leq q, i \in \mathcal{O}_{k_1}, j \in \mathcal{O}_{k_2},$ we have 
\begin{equation}\label{eveout_non}
 \langle \mathbf{u}_{{k_1, i}}, \mathbf{v}_{k_2,j} \rangle^2  \prec  n^{-1}.
\end{equation}

%
\end{thm}

Then we state the results of non-outlier eigenvectors. Denote 
\begin{equation*}
\theta_{k,i}:=n^{-2/3}(i \wedge (n_k+1-i))^{2/3}.
\end{equation*}

\begin{thm}\label{thm_evebulk} Fix $\tau>0.$ Suppose Assumption \ref{assum_main} holds. For $k=1,2,\cdots, q $ and $r_k^++1 \leq i \leq (1-\tau) n_k,$ 
we have that  
%
\begin{equation}\label{delocalequation}
\langle \mathbf{u}_{k,i}, \vb_{j} \rangle^2 \prec
n^{-1} ( \theta_{k,i}+(-\tsig_{j}^{-1}-x_{2k-1})^2)^{-1}, \ 1 \leq j \leq p.
\end{equation}
\end{thm}
{
 Theorems \ref{thm_eveout} and \ref{thm_evebulk}  characterize the asymptotic behavior of sample eigenvectors, whose analogous results are \cite[Theorems 2.11, 2.16 and 2.17]{BKYY}. } For the $(k,i)$th outlier eigenvectors, it will be concentrated on a cone with axis parallel to $\vb_{k,i}$ and the aperture is determined by $a_{k,i}.$ For the non-spiked eigenvector, they will be delocalized according to (\ref{delocalequation}). Moreover, {we conclude from (\ref{delocalequation}) that  the convergent rates of the extremal non-outlier eigenvectors near the edges of the bulk components depend on the separation between the spikes and the bulk components.}

{Before concluding this section, we point out how our model and results differ from some existing works. In \cite{BY}, the authors studied both the first and second order asymptotics of the outlier eigenvalues under the assumption that $\widetilde{\Sigma}$ has a block structure and only one bulk component. They also need stronger assumption that $-\widetilde{\sigma}_{i}^{-1}$ is far away from $x_1$ by a distance of order one (i.e. supercritical condition).   In  \cite{BGN}, borrowing the techiniques from free probability theory, the authors established the convergent limits for the eigenvalues and eigenvectors under the  supercritical condition. In \cite{DP}, the author obtained both the first and second order asymptotics for the eigenvalues and eigenvectors when $\Sigma=I$ under the supercritical condition. Finally, in \cite{BKYY}, the authors established both the first and second order asymptotics without the supercritical condition when $\Sigma=I$ and they also allow $\widetilde{\sigma}_i$ diverge with $n.$  }
\subsection{Examples} We consider a few examples to explain our results in details. We first provide two types of conditions on $\Sigma$ satisfying  Assumption \ref{assu_regularity}. They can be found in \cite[Examples 2.8 and 2.9]{KY1}.
{
\begin{con}\label{condition1}  Suppose that $l$ is fixed and there are $l$ distinct eigenvalues $\sigma_1, \cdots, \sigma_l$. We further assume that $\sigma_1, \cdots, \sigma_l$ and $c^{-1} \pi(\{\sigma_1\}), \cdots, c^{-1} \pi(\{\sigma_l\})$ all converge in $(0, \infty)$ as $n \rightarrow \infty$. We also assume that the critical points of  $\lim_{n} f$ are non-degenerate, and $\lim_{n} a_i > \lim_{n}a_{i+1}$ for $i=1,2,\cdots, 2q-1$.  
\end{con}
}
\begin{con}\label{condition_2} We suppose that $c \neq 1$ and $\pi$ is supported in some interval $[a, b] \subset(0, \infty),$ and that $\pi$ converges weakly to some measure $\pi_{\infty}$ that is absolutely continuous and whose density satisfies  that $\tau \leq d \pi_{\infty}(E)/dE \leq \tau^{-1}$ for $E \in [a,b].$ In this case, $q=1.$
\end{con}

 In all the examples below, we only derive the results for the eigenvalues and leave the discussion and interpretation of the eigenvectors to the readers.  We first provide two examples satisfying Condition \ref{condition1}.
\begin{exam}[Johnstone's spiked  covariance matrix model, i.e. BBP transition \cite{BBP}] \label{exam_classicrankone}  Let $\sigma_i=1, \ i=1,2,\cdots, p.$  We suppose that $r=1$ { and $\widetilde{\sigma}_1=1+d$  with $d \equiv d_1$ in the following discussion.} Since  $f(x)=-x^{-1}+c^{-1}(x+1)^{-1},$ it can be easily checked that the critical points of $f(x)$ are $-\frac{\sqrt{c}}{\sqrt{c}-1}, \ -\frac{\sqrt{c}}{\sqrt{c}+1},$ which implies that $q=1.$ By (\ref{outlier_out}), the convergent limit of the largest eigenvalue is  $\mu=f(-(d+1)^{-1})=1+d+c^{-1}(1+d^{-1}),$ and the phase transition happens when $d > c^{-1/2}.$ Furthermore,  the local convergence result reads as 
\begin{equation*}
\left|\mu_1-f(-\frac{1}{d+1})\right| \prec n^{-1/2}(d-c^{-1/2})^{1/2},
\end{equation*}
which agrees with \cite[Theorem 2.3]{BKYY}.
\end{exam}
\begin{exam} [Spiked model with variance clusters]\label{exam_multibulk} Consider that 
\begin{equation*} 
\widetilde{\Sigma}=\operatorname{diag}\{35, \underbrace{18, \cdots,  18}_\text{$p/2-1$  \ times}, 4, \underbrace{1, \cdots, 1}_\text{$p/2-1$ \ times}\}.
\end{equation*}
For detailed computation,  we set $c=2.$ Since $f(x)=-x^{-1}+0.25\big((x+\frac{1}{18})^{-1}+(x+1)^{-1}\big),$ we find that the critical points are  approximately $-2.39, -0.626,$ $-0.11, -0.037.$ Hence, $q=2.$  Due to the fact that
\begin{equation*}
f(-\frac{1}{35})>f(-0.037), \ f(-\frac{1}{4})>f(-0.626),
\end{equation*}
we find that there are two outlier eigenvalues  and they will locate around $f(-\frac{1}{35}), f(-\frac{1}{4})$. The local convergent results can be derived similarly. 

\end{exam}
Next we provide two examples satisfying Condition \ref{condition_2}, where there exists only one bulk component. 
\begin{exam}[Spiked model with uniformed distributed eigenvalues]\label{exam_uni} Consider that 
\begin{equation*}
\widetilde{\Sigma}=\operatorname{diag}\{8, 2.9975, 1.995, \cdots, 1.005, 1.0025 \}.
\end{equation*}
The limiting spectral distribution of $\Sigma$ is the uniform distribution on the interval $[1,3].$ Let $c=2$ and we have 
$f(x)=-(2x)^{-1}-(4x^2)^{-1} \log( (3x+1)(x+1)^{-1}).$ Its critical points are approximately $-2.005 ,\ -0.25.$ Therefore, the left and right edges are 
\begin{equation*}
f(-2.005) \approx 0.1494, \ f(-0.25) \approx 6.3941.
\end{equation*}
Hence, there exists one outlier eigenvalue since $f(-\frac{1}{8})>6.3941.$ { We remark that we are using the limits of $\pi$ and $f$ in our computation since when $p$ becomes larger,  $\pi$ (Recall (\ref{def_pi})) will converge to its limit quickly. }
\end{exam}

\begin{exam} [Spiked Toeplitz matrix]\label{exam_top} Suppose that $\Sigma$ is a Toeplitz matrix whose $(i,j)$-th entry is $0.4^{|i-j|}$ and the largest eigenvalue of $\widetilde{\Sigma}$ is $10$. We choose $c=2$ and $f$ can be written as
\begin{equation*}
f(x)=-\frac{1}{2x}-\frac{1}{3.81x^2} \log \left(\frac{2.332x+1}{0.4286x+1}\right),
\end{equation*}
where the interval $[0.4286, 2.332]$ is approximately the support of the population eigenvalues. The critical points of $f$ are approximately $-0.33, -3.62.$ Therefore, the left and right edges are 
\begin{equation*}
f(-3.62) \approx 0.0859, \ f(-0.33) \approx 4.3852.
\end{equation*}
There exists one spiked eigenvalue since $f(-\frac{1}{10}) >4.3852.$ 
\end{exam}

{ Before concluding this section, we discuss how Examples \ref{exam_classicrankone}--\ref{exam_top} follow from some previous works in the literature.  First all, in \cite{BGN}, the authors derived the convergent limits using the framework of free probability theory. In this sense, all the convergent limits in the above examples can be computed using the results of \cite{BGN}. However, in practice, when $\Sigma \neq I,$ it is generally difficult to write down the formulas explicitly using free probability theory. Our results provide a more practical way to write down the results.  Moreover, there are no results on the convergent rates in \cite{BGN}. Second, Examples \ref{exam_classicrankone}--\ref{exam_uni} follow from \cite{BY} since the population covariance matrices are diagonal in these examples. Finally, Example \ref{exam_classicrankone} follows from some previous works on Johnstone's spiked covariance matrix model, for instance \cite{BBP, BKYY, DP}. }

\section{Some remarks on statistical applications}\label{section_exapp}


This section is devoted to discussing the statistical applications of our results. For the ease of discussion, we consider the following case by strengthening (\ref{defn_oplus}) to
\begin{equation*}
\mathcal{O}_k'=\{i: x_{2k-1} + \varsigma \leq-\tilde{\sigma}_{k,i}^{-1}<x_{2(k-1)}\}, \ k=1,2, \cdots, q,
\end{equation*} 
for some small constant $\varsigma>0.$  {We assume that all the $r$ spikes belong to such sets.  Till the end of the discussion of this section, we also assume that there exists some constant $\varsigma>0$ so that for all $(k_1,i), (k_2,j) \in \mathcal{I},$
\begin{equation*}
|\widetilde{\sigma}_{k_1,i}-\widetilde{\sigma}_{k_2,j}| \geq \varsigma, \ i \neq j. 
\end{equation*}
The above assumption together with (\ref{defnregularityequation1}) guarantee that the outlier eigenvalues will be well-separated from each other.
}

Since our model is an extension of Johnstone's spiked covariance matrix model, {we can extend the statistical results in a more general setting.} For instance, the estimation of spiked eigenvalues \cite{BD}, the detection of number of spikes \cite{PY2012},  the spectrum estimation \cite{NK}, the eigenvector estimation \cite{MV},  and the estimation of high dimensional covariance matrices \cite{CRZ,DGJ, FLM, LW}. We now focus on a few concrete applications to illustrate our results. 

\subsection{Estimation of {the} number and localtion of spikes for $\widetilde{\Sigma}$} \label{sec:estnumberclsuter}   

An important parameter for estimation in our model is the number of spikes.  In practice, it may have important meanings, for instance the number of signals \cite{KN} in signal detection, the number of factors in factor model \cite{BN} and the number of clusters \cite{JW} in cluster analysis. 

In \cite{PY2012}, the authors proposed an algorithm for estimating the number of spikes $r$ assuming (\ref{eq_sigmaoriginal}).  They employ the differences between consecutive eigenvalues as their statistic, i.e., for a carefully determined threshold $\mathfrak{d}=O_{\prec}(n^{-2/3})$ and a fixed constant $s,$ they use the statistic,
\begin{equation}\label{eq_rhat}
\hat{r}=\max \{1 \leq j \leq s:  \delta_j \geq \mathfrak{d} \ \text{and} \ \delta_{j+1}<\mathfrak{d} \}, \ \delta_j=\mu_j-\mu_{j+1}. 
\end{equation}  
The above method can be only applied when all the spikes are  associated with the first bulk component. To address this issue, we shall first estimate the number of bulk components $q$ using $q^*$ defined as 
\begin{equation}\label{q*}
q^*=|\{1 \leq j \leq p \wedge n:  \delta_j \leq \mathfrak{d} \ \text{and} \ \delta_{j+1}>\mathfrak{d} \}|.
\end{equation}
After getting the estimate $q^*,$ we can relabel all the sample eigenvalues using $q^*.$ Next, we estimate the number of spikes as 
{
\begin{equation}\label{r*}
r^*=\sum_{k=1}^{q^*} r_k^*, 
\end{equation}
}
where $r_k^*$ is defined as 
\begin{equation*}
r_k^*=\max \{1 \leq j \leq s:  \delta^k_j \geq \mathfrak{d} \ \text{and} \ \delta^k_{j+1}<\mathfrak{d} \}, \ \delta^k_j=\mu_{k,j}-\mu_{k, j+1}. 
\end{equation*} 

As we can see from the definitions of $r_k^*$ and $q^*,$ the sharp transition from a larger difference to a smaller one determines the number of spikes and the opposite determines the number of bulk components. Further, we can also get the location of spikes 
\begin{equation}\label{eq_location}
\{(k,i): 1 \leq k \leq q^*, 1 \leq
 i \leq r_k^*\},
 \end{equation}
 according to our statistic. Following the proof of \cite[Theorem 1]{PY2012}, it is easy to see that 
\begin{equation*}
\mathbb{P}(r^*=r) \rightarrow 1, \ n \rightarrow \infty. 
\end{equation*}

As a concrete application of the new estimator $r^*,$  we consider the estimation of covariance matrix using factor model, where  the key inputs are the number and locations of the factors.  In the lierature of financial economics, stock returns are modeled using a few common factors \cite{FLM} 
\begin{equation}\label{defn_factor}
Y_{it}=\mathbf{b}_i^\T \mathbf{f}_t+u_{it}, t=1,2,\cdots, n,
\end{equation} 
where $Y_{it}$ is the return of the $i$-th stock at time $t,$ $\mathbf{b}_i$ is a vector of factor loadings, $\mathbf{f}_t$ is a $r \times 1$ vector of latent common factors and $u_{it}$ is the idiosyncratic component, which is uncorrelated with $\mathbf{f}_t.$  The matrix form of (\ref{defn_factor}) can be written as $\mathbf{Y}_t=\mathbf{B}\mathbf{f}_t+\mathbf{u}_t.$ For the purpose of identifiability, we impose the following constraints \cite{FLM}: $\operatorname{Cov}(\mathbf{f}_t)=\mathbf{I}_r$ and the columns of $\mathbf{B}$ are orthogonal. As a consequence, the population covariance matrix can be written as 
\begin{equation} \label{defn_factor_cov}
\widetilde{\Sigma}=\mathbf{B}\mathbf{B}^\T+\bm{\Sigma}_u.
\end{equation}
{We further assume that $\bm{\Sigma}_u$ is diagonal. Therefore,} (\ref{defn_factor_cov}) can be represented using (\ref{sigma_mostgene1}).  The estimation can be computed using the constraint least square optimization \cite{FLM, SW}
\begin{align*}
&\operatorname{arg} \min_{\mathbf{B, F}} ||\mathbf{Y}-\mathbf{BF^\T} ||_F^2, \\
& n^{-1}\mathbf{F}^\T \mathbf{F}=I_r, \ \mathbf{B^\T B} \ \text{is diagonal}. \nonumber
\end{align*}
The least square estimator for $\mathbf{B}$ is $\widehat{\mathbf{\Lambda}}=n^{-1}\mathbf{Y} \widehat{\mathbf{F}},$ where the columns of $\widehat{\mathbf{F}}$ satisfy that $n^{-1/2}\widehat{F}_i$ is the eigenvector corresponds to the $i$th largest eigenvalue of $\mathbf{Y}^\T \mathbf{Y}.$ In \cite{FLM}, the authors showed that $\mathbf{B}\mathbf{B}^\T$ corresponded to the spiked parts whereas $\bm{\Sigma}_u$ the non-spiked part. In most of the applications, $\bm{\Sigma}_u$ is assumed to have some sparse structure. Hence, the estimator can be written as $\widehat{\mathbf{\Sigma}}=\widehat{\mathbf{\Lambda}}_r\widehat{\mathbf{\Lambda}}_r^*+\widehat{\mathbf{\Sigma}}_u,$ where $\widehat{\mathbf{\Lambda}}_r$ is the collection of the columns correspond to the factors and $\widehat{\mathbf{\Sigma}}_u$ is the estimation using some thresholding method by analyzing the residual.

\begin{figure}[!ht]
\begin{center}
  \includegraphics[height=5cm, width=10cm]{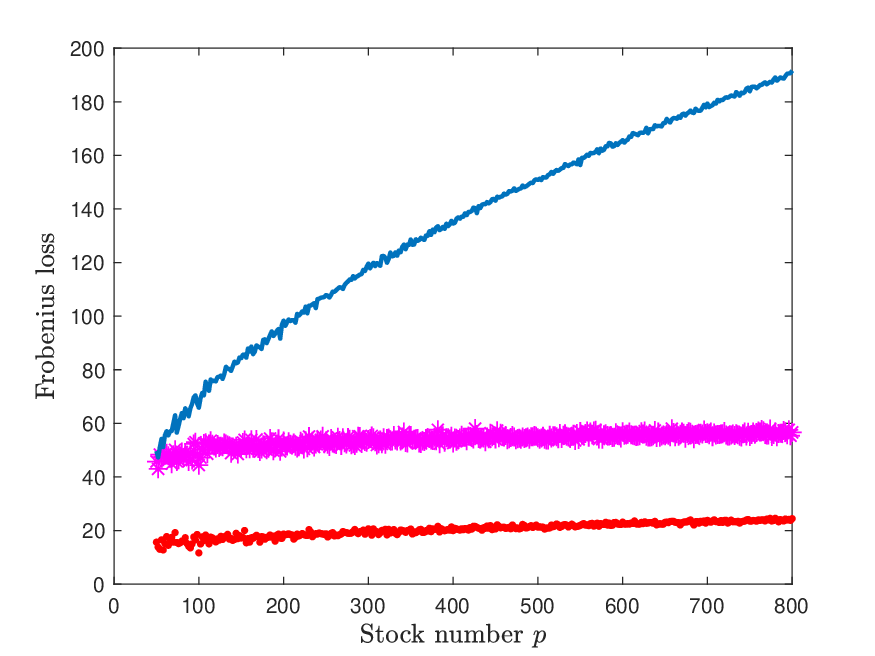} 
  \end{center}
  \vspace*{-6mm}
\caption{Estimation loss using factor model { for $\widetilde{\Sigma}$ in (\ref{defn_factor_cov})}. We simulate the estimation error  { for $\widetilde{\Sigma}$} under Frobenius norm using Example \ref{exam_multibulk}, where the blue line stands for the sample covariance matrix estimation, red dots for our Multi-POET estimation and magenta dots for the POET estimation. We find that our new estimator can help to improve the estimation accuracy. }
 \label{figmultipoet}
\end{figure}

In \cite{FLM}, the authors use the Principal Orthogonal complEment Thresholding (POET) method to numerically compute the estimator. In their setting, they estimate $r$ by solving a numerical optimization problem from \cite{BN}, which is similar to the idea of using $\hat{r}$ in (\ref{eq_rhat}).  However, in practice, due to volatility cluster \cite{VBD}, such factors can locate anywhere. Therefore, we extend POET to a multiple version (Multi-POET) by using $r^*$ in (\ref{r*}) and the locations (\ref{eq_location}). For instance, consider $\widetilde{\Sigma}$ from  Example \ref{exam_multibulk}, there are two factors and the indices of their locations are $1, \frac{p}{2}+1.$   Figure \ref{figmultipoet} shows that our results can help to reduce the estimation loss in  such situation. {We remark that in \cite{FLM}, they assume that the spikes $\widetilde{\sigma}_{k,i}, \ (k,i) \in \mathcal{I}$ are of order $O(p).$ In this case, all the outlier eigenvalues will be associated with the first bulk component and hence POET will work. However, when the spikes are of the same order as the bulk population eigenvalues, for instance in the study of neural activity \cite{BMCP, LSBT}, we need to use Multi-POET.}

\subsection{{Nonlinear shrinkage estimation} for $\widetilde{\Sigma}$: the spiked case} In many situations, we have no information  on the true eigenvectors of $\widetilde{\Sigma}.$  A natural choice for us is to use the sample eigenvectors $\{\ub_{i}\}$. Given some loss function $\mathcal{L}(\cdot, \cdot),$ we need to find a diagonal matrix $\widehat{\Lambda},$ such that $\mathcal{L}(\widetilde{\Sigma}, \mathbf{U} \widehat{\Lambda}\mathbf{U}^\T)$ is minimized, where $\mathbf{U}=(\mathbf{u}_i).$ Then our estimator will be $\widehat{\Sigma}=\mathbf{U}\widehat{\Lambda}\mathbf{U}^\T$. Such an estimator is oracle since it achieves suboptimality when we restrict ourselves on the sample eigenvectors. Some special cases have been considered in \cite{Jbun, DGJ}. In this section, we revisit this problem using our proposed model. 

 We consider a special case when $\widehat{\Lambda}$ has a closed-form solution by setting 
\begin{equation*}
 \mathcal{L}(\widetilde{\Sigma}, \widehat{\Sigma})=\norm{\widetilde{\Sigma}-\widehat{\Sigma}}_F^2, \ \widehat{\Sigma}=\mathbf{U}\widehat{\Lambda}\mathbf{U}^\T.
\end{equation*}
For the ease of discussion, we assume that $q=1$ and hence we can rewrite $\mathcal{I}=\{i: 1 \leq i \leq r\}.$ First of all, we have 
\begin{equation}\label{oracle_decom}
\norm{\widetilde{\Sigma}-\widehat{\Sigma}}_F^2=\left(\norm{\widetilde{\Sigma}}^2_F-\norm{\operatorname{diag}(\mathbf{U}^\T \widetilde{\Sigma}\mathbf{U})}_F^2\right)+\norm{\widehat{\Lambda}-\operatorname{diag}(\mathbf{U}^\T \widetilde{\Sigma}\mathbf{U})}_F^2.
\end{equation}
Since the first part of the right-hand side of (\ref{oracle_decom}) cannot be optimized,  we should take $\widehat{\Lambda}=\operatorname{diag}(\mathbf{U}^\T \widetilde{\Sigma} \mathbf{U})=\text{diag}\{\widehat{\lambda}_1, \cdots, \widehat{\lambda}_p\},$ where    
\begin{equation}\label{oracle_est_decom}
\widehat{\lambda}_i =\mathbf{u}^\T_i  \widetilde{\Sigma} \mathbf{u}_i, \ i=1,2,\cdots, p. 
\end{equation}   
%
%
We consider the estimation of $\widehat{\lambda}_i$ for $i \in \mathcal{I}.$  First of all, we observe that (\ref{oracle_est_decom}) can be written as 
\begin{equation}\label{oracel_diag_decom}
\widehat{\lambda}_i=\sum_{j \in \mathcal{I}} \tsig_j\langle \mathbf{v}_j, \mathbf{u}_i \rangle^2+\sum_{j \notin \mathcal{I}} \tsig_j \langle \mathbf{v}_j, \mathbf{u}_i \rangle^2.
\end{equation}
For $i \in \mathcal{I},$ by (\ref{eveout_eq}), the first part of the right-hand side of (\ref{oracel_diag_decom}) satisfies
\begin{equation}\label{RIE1}
\sum_{j \in \mathcal{I}} \tsig_j \langle \mathbf{v}_j, \mathbf{u}_i\rangle^2 \rightarrow \frac{f^{\prime}(-\tsig_i^{-1})}{f(-\tsig^{-1}_i)}, \ n\rightarrow \infty. 
\end{equation}
For the second part of (\ref{oracel_diag_decom}),   by (4.16) of \cite{BBP2017}, we have 
\begin{equation} \label{RIE2}
\sum_{j \notin \mathcal{I}} \tsig_j \langle \mathbf{v}_j, \mathbf{u}_i \rangle^2 \rightarrow \frac{1}{n} \frac{\tsig_j^2}{f(-\tsig_j^{-1})} \sum_{j=1}^p \frac{\sigma_j^2}{(\tsig_i-\sigma_j)^2},  \ n \rightarrow \infty. 
\end{equation}
Meanwhile, inserting $f(-\tsig_i^{-1})$ back into (\ref{defnf}), we get
\begin{equation}\label{oracl_insetbackeq}
f(-\tsig_i^{-1})=\tsig_i+\frac{1}{n} \sum_{j=1}^p \frac{1}{-\tsig_i^{-1}+\sigma_j^{-1}}.
\end{equation}
Differentiating with respect to $\tsig_i$ on both sides of (\ref{oracl_insetbackeq}), we get
\begin{equation}\label{RIE3}
f^{\prime}(-\tsig_i)(\tsig_i)^{-2}=1-\frac{1}{n}\sum_{j=1}^p \frac{\tsig_j^2}{(\tsig_i-\sigma_j)^2}.
\end{equation} 
Therefore, by (\ref{oracel_diag_decom}), (\ref{RIE1}), (\ref{RIE2}) and (\ref{RIE3}), when $i \in \mathcal{I}, $ we have that
{
\begin{equation}\label{oracle_out}
\widehat{\lambda}_i \rightarrow \frac{(\widetilde{\sigma}_i)^2}{f(-1/\widetilde{\sigma}_i)}, \ n \rightarrow \infty.  
\end{equation}
(\ref{oracle_out}) needs an estimate of $\widetilde{\sigma}_i$ and the estimation of $\widetilde{\sigma}_i$ has been studied in \cite[Theorem 3.5]{DY2}. We record the estimator for the reader's convenience. We use $\widehat{\sigma}_i, \ i \in \mathcal{I}$ to estimate $\widetilde{\sigma}_i,$  where 
\begin{equation*}
\widehat{\sigma}_i=-\left(\frac{1}{n} \sum_{j \notin \mathcal{I}} \frac{1}{\mu_j-\mu_i}\right)^{-1}.
\end{equation*}

}

{We remark that in \cite{LW}, the authors proposed explicit formulas for $\widehat{\lambda}_i, i=1,2,\cdots,p,$ in (\ref{oracle_est_decom})  when $\widetilde{\Sigma}$ has no spikes. Indeed, this relates to the estimation of the non-outlier part $\widehat{\lambda}_i, \ i \notin \mathcal{I}$ in our setting. In this section, we have derived the formulas for $\widehat{\lambda}_i, \ i \in \mathcal{I}$ and we will consider  the estimation of $\widehat{\lambda}_i, \ i \notin \mathcal{I}$ in the future work.   }

\section{Preliminaries}\label{section_tools}
This section is devoted to introducing the basic tools for our proofs: the anisotropic local laws and perturbation arguments. 
\subsection{Anisotropic local law} In this section, we collect the results of the anisotropic local laws from \cite{KY1}.  We first summarize the properties of $m(z)$ defined in (\ref{defnf}). Its proof can be found in \cite[Lemmas A.4 and A.5]{KY1}.  For fixed small constant $\tau>0$ and $z=E+\ri \eta,$ we define the spectral domains
\begin{equation} \label{full}
\mathbf{S} \equiv \mathbf{S}(\tau, n):=\{z \in \mathbb{C}^{+}: |z| \geq \tau, |E|\leq \tau^{-1},  n^{-1+\tau} \leq \eta \leq \tau^{-1}\}.
\end{equation}
Further, for some constant $\tau'>0,$ we define 
\begin{equation}\label{edge}
\mathbf{S}^e_i \equiv \mathbf{S}^e_i(\tau^{\prime}, \tau, n):=\{z \in \mathbf{S}: E \in [a_{i}-\tau^{\prime}, a_{i}+\tau^{\prime}]\}, \ i=1,2,\cdots, 2q.
\end{equation}
{ For any two quantities $a_n$ and $b_n,$ we write $a_n \sim b_n$ if there exist two positive constants $C_1$ and $C_2$ such that $C_1|b_n| \leq |a_n| \leq C_2 |b_n|.$}
Denote $\kappa(E):=\min_{1 \leq
 k \leq 2q}|E-a_k|,$ we have 
\begin{lem} \label{lem_propm} For $z \in \mathbf{S},$ we have 
\begin{equation} \label{propm_equ1}
\operatorname{Im} m(z) \sim 
\begin{cases}
\sqrt{\kappa+\eta}, & \textsl{$E \in \operatorname{supp}(\rho)$ }; \\
\frac{\eta}{\sqrt{\kappa+\eta}}, & \text{E $\notin \operatorname{supp}(\rho)$}.
\end{cases},
\end{equation}
and 
\begin{equation}\label{propm_sta}
\min_i|m(z)+\sigma_i^{-1}| \geq \tau,
\end{equation}
where $\tau>0$ is some constant. Furthermore, if $z \in \mathbf{S}^e_i,$ we have
\begin{equation}  \label{propm_equ2}
| m(z)-x_i|\sim \sqrt{\kappa+\eta}.
\end{equation}
\end{lem}
Next we introduce the anisotropic local laws. Let $\mathcal{Q}_2=X^\T \Sigma X.$ Recall that the ESD of $\CQ_2$  is defined as
\begin{equation*}
F^{(n)}_{\CQ_2}(\lambda):=\frac{1}{n} \sum_{i=1}^n \mathbf{1}_{\{\lambda_i(\CQ_2) \leq \lambda\}}.
\end{equation*}
Further,  the Stieltjes transform of the ESD of $\CQ_2$ is given by
\begin{equation}\label{defn_stiel}
m_n(z):=\int \frac{1}{x-z}dF^{(n)}_{\CQ_2}(x)=\frac{1}{n} \mathrm{Tr} \CG_2 (z), \ z \in \mathbb{C}^+, 
\end{equation}
where  $\CG_2(z):=(\CQ_2-z)^{-1}.$  The first lemma collects the results when $z \in \mathbf{S},$ which is \cite[Theorem 3.6]{KY1}. Denote the deterministic control parameter
\begin{equation*}
\Psi(z):=\sqrt{\frac{\operatorname{Im}m(z)}{n\eta}}+\frac{1}{n\eta},
\end{equation*}
and the Green function of $\CQ_1$ as $\CG_1(z).$ 
\begin{lem} \label{lem_anisotropic} Suppose (\ref{eq_regemic}), (\ref{eq_momentx12}), (\ref{eq_mimentxhigh}), (\ref{assm_11}) and Assumption \ref{assu_regularity} hold true. Fix $\tau>0$, for $z \in \mathbf{S}$ and any unit vectors $\mathbf{u}, \mathbf{v} \in \mathbb{R}^{n}$ and $\mathbf{u}_1, \mathbf{v}_1 \in \mathbb{R}^{p}$  
\begin{equation*}
 \langle \mathbf{u},  (\CG_2(z)-m(z))\mathbf{v} \rangle \prec \Psi(z), \  \Big \langle \mathbf{u}_1,  \Big(\CG_1(z)-\frac{-1}{z(1+m(z) \Sigma )} \Big)\mathbf{v}_1 \Big \rangle \prec \Psi(z),
\end{equation*}
and $$m_n(z)-m(z) \prec \frac{1}{n \eta}.$$
\end{lem}
Outside the support of the asymptotic spectrum, we have stronger control all the way down to the real axis, which is \cite[Theorem 3.7]{KY1}.  

\begin{lem}\label{lem_anisoout} Under the assumptions of Lemma \ref{lem_anisotropic}, for $z \in [\tau, \tau^{-1}] \times (0, \tau^{-1}]$ satisfying $\text{dist}(E, \operatorname{supp}(\rho)) \geq n^{-2/3+\tau},$ we have 
\begin{align*}
&\langle \ub, (\CG_2(z)-m(z)) \vb \rangle \prec \sqrt{\frac{\im m(z)}{n \eta}},  \\
&  \Big \langle \mathbf{u}_1, 
  \Big(\CG_1(z)-\frac{-1}{z(1+m(z) \Sigma )} \Big)\mathbf{v}_1 \Big \rangle \prec \sqrt{\frac{\im m(z)}{n \eta}}. 
\end{align*}
\end{lem}

Finally, we record the results on the rigidity of eigenvalues. It is an important consequence of Lemma \ref{lem_anisotropic} and proved in \cite[Theorem 3.12]{KY1}.  For $k=1,2,\cdots,q$ and $i=1,2,\cdots, n_k,$ we relabel the eigenvalues for $\mathcal{Q}_1$ by denoting
\begin{equation} \label{defn_relabel}
\lambda_{k,i}:=\lambda_{i+\sum_{l<k} n_l}. 
\end{equation}
Further, we define the \emph{classical eigenvalue locations} of $\rho$ by   $\gamma_1 \geq \gamma_2 \geq \cdots \geq \gamma_{n \wedge p},$ where  

$$n \int_{\gamma_i}^{\infty} d \rho=i-\frac{1}{2}.$$

Similar to the relabellings (\ref{defn_relabel}), we denote 
\begin{equation*}
\gamma_{k,i}:=\gamma_{i+\sum_{l<k} n_l} \in (a_{2k}, a_{2k-1}).
\end{equation*}
It is notable that $\gamma_{k,i}$ can also be characterized through $n \int_{\gamma_{k,i}}^{a_{2k-1}}d\rho=i-\frac{1}{2}.$

\begin{lem} \label{lem_rigidedgeuni} Fix $\tau>0.$ Suppose the assumptions of Lemma \ref{lem_anisotropic} hold. For all $k=1,\cdots, q$ and $i=1,\cdots, n_k$ satisfying $\gamma_{k,i} \geq \tau,$ we have 
\begin{equation} \label{lem_rigi_equ}
|\lambda_{k,i}-\gamma_{k,i}| \prec n^{-2/3} (i \wedge(n_k+1-i))^{-1/3}.
\end{equation}
\end{lem}

\subsection{Perturbation analysis} In this section, we provide the perturbation identities. They provide the natural connection  between the locations of eigenvalues and eigenvectors with the Green function $\CG_1(z).$  To ease the discussion of the proof, till the end of the paper, we use the following parameterization for (\ref{sigma_mostgene1}).  { Moreover, from now on till the end of the paper, we will also use $\mathcal{I}$ as the index set of the spikes before relabellings. }  We denote {
\begin{equation} \label{sigma_mostgene}
\widetilde{\Sigma}=\sum_{i=1}^p \tilde{\sigma}_i \mathbf{v}_i \mathbf{v}_i^\T, \ \text{where} \ \tilde{\sigma}_i=(1+d_i) \sigma_i,
\end{equation} 
}
and 
{
\begin{equation*}
d_i>0, \ i \in \mathcal{I}, \ \text{and} \ d_i=0, \ i \notin \mathcal{I}. 
\end{equation*}
} 
We next rewrite $\widetilde{\Sigma}$ for future convenience.  Denote the diagonal matrix $\widetilde{D}=\text{diag}\{\tilde{d}_1, \cdots, \tilde{d}_p\}.$ With the above notations, we can rewrite 
\begin{equation} \label{sig_express}
\widetilde{\Sigma}=\Sigma(1+\mathbf{V}\widetilde{D}\mathbf{V}^\T)=(1+\mathbf{V}\widetilde{D}\mathbf{V}^\T) \Sigma,
\end{equation}
where $\mathbf{V}=(\mathbf{v}_i)_{i=1}^p.$  Furthermore, denote the $r \times r$ matrix $D_o$ as the diagonal matrix whose nonzero entries being { $\{d_i\}_{i \in \mathcal{I}}$, where we recall that we also use $\mathcal{I}$ as the index set of the spikes before relabellings.} Indeed, we have that 
\begin{equation} \label{invetibletransfer}
\mathbf{V}\widetilde{D} \mathbf{V}^\T=\sum_{i \in \mathcal{I}} d_i \mathbf{v}_i \mathbf{v}_i^\T=\mathbf{V}_o D_o \mathbf{V}_o^\T,
\end{equation}
where $\mathbf{V}_o$ is a $p \times r$ matrix containing $\mathbf{v}_i, \ i \in \mathcal{I}.$  For the convenience of the statements,  we introduce more relabellings. For $k=1,2,\cdots, q$ and $i=1,2,\cdots, r_k,$ we next introduce the relabelings for { $\{d_i\}_{i \in \mathcal{I}}$ } by 
{
\begin{equation*}
d_{k,i}=d_{i+\sum_{l<k} n_l}. 
\end{equation*}
} 
 With the above convention, we can relabel $\{\tilde{\sigma}_i, \ i \in \mathcal{I}\}$
\begin{equation}\label{eq_relabelsigma}
\tilde{\sigma}_{k,i}=\sigma_{k,i}(1+d_{k,i}).
\end{equation}

Armed with the above preparation, we now state some lemmas. The counterparts of these lemmas can be found in \cite[Section 3.4 and Lemma 4.1]{BKYY} and we omit the proofs here. 
\begin{lem} \label{lem_pertubation}
For $\mu \in \mathbb{R},$ it is an eigenvalue of $\CTQ_1$ but not $\CQ_1$ if and only if 
\begin{equation}\label{cor_pertueq}
\det(D_o^{-1}+1+\mu \mathbf{V}^\T_o \CG_1(\mu) \mathbf{V}_o)=0.
\end{equation}
\end{lem} 


%

Next we will extend the Weyl's interlacing theorem to fit our setting.
%



\begin{lem} \label{cor_interlacing} Recall the model defined in (\ref{sigma_mostgene}). For $k=1,2,\cdots,q,$ we have
\begin{equation*}
\mu_{s_k+i} \in [\lambda_{k,i+r_k}, \lambda_{k, i-r_k}], \ 1 \leq i \leq n_k,
\end{equation*}
where $s_k=\sum_{i=1}^{k-1}n_i$ and we use the convention that $\lambda_{k,i-r_k}:=\infty$ when $i-r_k<1$ and $n_0=0.$
\end{lem}
The following lemma establishes the connection between the Green functions of $\CG_1$ and $\CTG_1.$  It serves as the key identity for analyzing the eigenvectors. 

\begin{lem} \label{lem_vector}
For $z \in \mathbb{C}^+,$ we have 
\begin{equation*}
\mathbf{V}_o^\T \CTG_1(z) \mathbf{V}_o=\frac{1}{z} \Big [ D_o^{-1}-\frac{(1+D_o)^{1/2}}{D_o}(D_o^{-1}+1+z \mathbf{V}_o^\T \CG_1(z)\mathbf{V}_o)^{-1}\frac{(1+D_o)^{1/2}}{D_o} \Big ].
\end{equation*}
\end{lem}

\section{Proof of Theorem \ref{thm_evout}}\label{section_value}
In this section, we study the convergent limits and rates of the eigenvalues of $\widetilde{\mathcal{Q}}_1$ and prove Theorem \ref{thm_evout}. The proof strategy is similar to those in \cite[Section 4]{BKYY} and we only sketch the proof here. 

Without loss of generality,  we  focus our discussion on the $k$th bulk component, $k=1,2,\cdots,q$.  Throughout the proof, we will make use of the subset of $\mathcal{O}_k$
\begin{equation}\label{eq_otau}
\mathcal{O}^{\epsilon}_k=\{i: x_{2k-1} + n^{-1/3+\epsilon}\leq-\tilde{\sigma}_{k,i}^{-1}<x_{2(k-1)}\}, \ k=1,2, \cdots, q,
\end{equation}
where $\epsilon>0$ is some fixed small constant. 
Regarding on (\ref{defn_oplus}), we find that $\mathcal{O}_k=\mathcal{O}_k^0.$ We will frequently use the following identity (Recall (\ref{key_quantity}))
\begin{equation}\label{determinstic_quant}
-\frac{1}{1+\sigma_{k,i} m(f(-\tilde{\sigma}_{k,i}^{-1}))}=-d_{k,i}^{-1}-1, \ i \in \mathcal{O}_k.
\end{equation}

\begin{proof}
Since the proofs of (\ref{outlier_out}) and (\ref{outlier_nonout}) are similar, we focus our discussion on (\ref{outlier_out}) and sketch the proof of (\ref{outlier_nonout}) in the end of the proof.   For $\epsilon>0$ and  $ i \in \mathcal{O}_k^{\epsilon},$  we denote $\rI_i \equiv \rI_i^k(\widetilde{D})$ by 
\begin{equation*}
\rI_i:=[f(-\tilde{\sigma}_{k,i}^{-1})-(-\widetilde{\sigma}_{k,i}^{-1}-x_{2k-1})^{1/2}n^{-1/2+\epsilon}, f(-\tilde{\sigma}_{k,i}^{-1})+(-\widetilde{\sigma}_{k,i}^{-1}-x_{2k-1})^{1/2}n^{-1/2+\epsilon}],
\end{equation*}
and 
\begin{equation*}
\rI_0:=[0, a_{2k-1}+n^{-2/3+2\epsilon}].
\end{equation*}
For convenience, we denote $\rI=\rI_0 \cup \bigcup_{i \in \mathcal{O}_k^\epsilon} \rI_i.$ 
By (\ref{propm_equ1}) and Lemma  \ref{lem_anisoout},  we conclude that for such $\epsilon>0,$ there exists a high-probability event   $\Xi \equiv \Xi(\tau, \epsilon,k)$ satisfying the following conditions. \\
(i). Denote $\kappa_k:=|x-a_{2k-1}|,$ for $ \  a_{2k-1}<x<a_{2(k-1)},$ we have
{
\begin{equation*}
\mathbf{1}(\Xi) \norm{\Vb_o^\T\Big (x \CG_1(x)-\frac{-1}{1+m(x) \Sigma} \Big ) \Vb_o} \leq \kappa_k^{-1/4} n^{-1/2+\epsilon/2}.
\end{equation*}
}
(ii). For $1 \leq i \leq 2r_k,$ we have 
\begin{equation*}
\mathbf{1}(\Xi)|\lambda_{k,i}-a_{2k-1}| \leq n^{-2/3+\epsilon}. 
\end{equation*}
Since $\mathbf{1}(\Xi) \lambda_{k,1} \leq a_{2k-1}+ n^{-2/3+\epsilon},$ we conclude from (\ref{cor_pertueq}) that $x \notin \rI_0$ is an eigenvalue of $\CTQ_1$ if and only if 
{
\begin{align*}
\mathbf{1}(\Xi)\Big(D_o^{-1}+1+x \Vb_o^\T \CG_1(x) \Vb_o \Big)=\mathbf{1}(\Xi) \Big(D_o^{-1}+1-\frac{-1}{1+m(x) \Lambda_o}+O(\kappa_k^{-1/4}n^{-1/2+\epsilon/2}) \Big),
\end{align*}
}
is singular, where $\Lambda_o$ is a diagonal matrix containing $\{\sigma_i\}_{i \in \mathcal{I}}.$ Here we recall that we have used $\mathcal{I}$ as the index set before relabellings. Denote $\mathcal{L}_i(x)$ by
\begin{equation*}
\mathcal{L}_i(x):=\left(D_o^{-1}+1-\frac{-1}{1+m(x) \Lambda_o} \right)_{ii}. 
\end{equation*}

Since $r_k=O(1),$ it suffices to prove the following lemma.
{
\begin{lem}
For $x \notin \rI,$  on the event $\Xi,$ we have that 
\begin{equation} \label{eq_keyequation}
\min_{i \in \mathcal{I}}\Big| \mathcal{L}_i(x)  \Big| \gg \kappa_k^{-1/4} n^{-1/2+\epsilon/2}.
\end{equation}
\end{lem}
}
The proof of (\ref{eq_keyequation}) relies on (\ref{fderivative})  and $f$ is monotone {increasing} outside the support of $\rho$. It is similar to that of \cite[eq. (4.6)]{BKYY} and we omit the details. This implies that on the event $\Xi,$ the complement of $\rI$ contains no eigenvalues of $\CTQ_1.$ 

Next, to prove (\ref{outlier_out}),  we will show that the neighborhood $\rI_i$ contains the right number of eigenvalues of $\CTQ_1.$ We will use a continuity argument similar to \cite[Sections 6.4 and 6.5]{KY2}. In a first step, we consider $\widetilde{D} \equiv \widetilde{D}(0)$ such that   {
\begin{equation}\label{eq_multi}
\widetilde{\sigma}_{1}>\widetilde{\sigma}_{2}>\cdots>\widetilde{\sigma}_{r}>0, \ \min_{i \neq j} |\widetilde{\sigma}_i-\widetilde{\sigma}_j| \geq \varsigma,
\end{equation}
where $\varsigma>0$ is some fixed constant. }
We will show that 
each interval $\rI_i(\widetilde{D}(0)), \ i \in \mathcal{O}_k^{\epsilon}$ contains precisely one eigenvalue of $\CTQ_1.$  We pick a small $n$-independent counterclockwise (positive-oriented) contour $\mathcal{C} \subset \mathbb{C} \backslash \bigcup_{k=1}^q[a_{2k}, a_{2k-1}]$ that encloses $f(-\widetilde{\sigma}_{k,i}^{-1})$ but no other points. For large enough $n,$ define 
\begin{equation*}
F(z):=\det(D_o^{-1}+1+z\mathbf{V}_o^{\T}\CG_1(z)\mathbf{V}_o), \ K(z):=\det(D_o^{-1}+1-\frac{-1}{1+m(z)\Lambda_o}).
\end{equation*}
It is easy to see that $F(z), K(z)$ are holomorphic on and inside $\mathcal{C}$ and $K(z)$ has precisely one zero $f(-\widetilde{\sigma}_{k,i}^{-1})$ inside $\mathcal{C}.$ On the contour $\mathcal{C}, $ by (\ref{propm_equ1}) and Lemma  \ref{lem_anisoout}, it is easy to check that on $\Xi,$  for some constant $\delta>0,$ 
\begin{equation*}
\min_{z \in \mathcal{C}} |K(z)| \geq \delta>0, \ |F(z)-K(z)|\leq N^{-1/2+\epsilon} \kappa^{-1/4}.
\end{equation*}
We hence conclude from Rouche's theorem that $\rI_i$ contains precisely one eigenvalue of $\CTQ_1$.  This proves (\ref{outlier_out}) under the assumption (\ref{eq_multi}). Next, we deduce the general case using a continuity argument by choosing a suitable continuous path $(\widetilde{D}(t))_{t \in [0,1]}$ connecting $\widetilde{D}(0)$ satisfying (\ref{eq_multi}) and $\widetilde{D}(1)$ without such assumption. We summarize the results as the following lemma and omit the proof. It can be found in \cite[Section 4]{BKYY}. 
{
\begin{lem} \label{lem_boostrap}  Under the assumptions of Theorem \ref{thm_evout}, we have that $\mu_{k,i} \in \rI_i,  \ i \in \mathcal{O}_k^{\epsilon}$  without the assumption (\ref{eq_multi}).    
\end{lem}
}
We then briefly discuss the proof of  (\ref{outlier_nonout}). It follows from a discussion similar to \cite[Proposition 6.5]{KY2}. We first prove for fixed configuration $\widetilde{D} \equiv \widetilde{D}(0)$ when  (\ref{eq_multi}) holds true. Since  $\mu_{k,i} \in \bigcup_{i \in \mathcal{O}_{k}^{\epsilon}} \rI_i, \ i \in \mathcal{O}_{k}^{\epsilon},$  we have 
\begin{equation*}
\mu_{k,i} \in \rI_0, \ i \in [r_k]/\mathcal{O}_k^{\epsilon}. 
\end{equation*}
As Assumption \ref{assu_regularity} holds true, the lower bound follows from the perturbation result Lemma \ref{cor_interlacing} and rigidity result Lemma \ref{lem_rigidedgeuni}. In a second step, we will use a discussion similar (actually easier) to Lemma \ref{lem_boostrap} to remove the assumption  (\ref{eq_multi}). We omit further details here. 

In summary, we have proved that for $\epsilon>0$ and $i \in \mathcal{O}_{k}^{4 \epsilon},$
\begin{equation}\label{eq_pp1}
\mathbf{1}(\Xi)|\mu_{k,i}-f(-\tilde{\sigma}^{-1}_{k,i})|\leq (-\widetilde{\sigma}_{k,i}^{-1}-x_{2k-1})^{1/2}n^{-1/2+\epsilon},
\end{equation} 
and for $i \in [|\mathcal{O}_k^{4\epsilon}|+1,|\mathcal{O}_k^{4 \epsilon}|+\varpi]$ and some constant $C>0,$
\begin{equation}\label{eq_pp2}
\mathbf{1}(\Xi)|\mu_{k,i}-a_{2k-1}| \leq C n^{-2/3+8\epsilon}.
\end{equation}
From (\ref{eq_pp2}) and (\ref{fderivative}), we find that for $i$ satisfying that $x_{2k-1}+n^{-1/3} \leq -\widetilde{\sigma}_{k,i}^{-1} \leq x_{2k-1}+n^{-1/3+4\epsilon},$ we have that 
\begin{align*}
\mathbf{1}(\Xi)|\mu_{k,i}-f(-\tilde{\sigma}_{k,i}^{-1})| & \leq \mathbf{1}(\Xi)(|\mu_{k,i}-x_{2k-1}|+|f(-\tilde{\sigma}_{k,i}^{-1})-x_{2k-1}|) \\
& \leq C (-\widetilde{\sigma}_{k,i}^{-1}-x_{2k-1})^{1/2}n^{-1/2+8\epsilon}.
\end{align*}
This concludes our proof. 
\end{proof}

\section{Proofs of Theorems \ref{thm_eveout} and \ref{thm_evebulk}}\label{section_vector}
\subsection{Proof of Theorem \ref{thm_eveout} under non-overlapping condition}\label{sec_spikedvector}
The proof is similar to that of \cite[Theorems 2.11 and 2.16]{BKYY} and we only sketch the proof here.
We focus our discussion on  (\ref{eveout_eq}). 
We divide our proofs into two steps.  We again focus our discussion on the $k$th bulk component.    First,  we will the results for $\mathcal{O}^{\tau}=\bigcup_{k=1}^q \mathcal{O}_k^{\tau},$ where $\mathcal{O}^{\tau}_k$ is defined in (\ref{eq_otau}) and $\tau>0$ is some fixed constant.  We  prove a proposition satisfying the following \emph{non-overlapping condition}.  
\begin{assu}\label{assum_stronger}
Fix some constant $\tau>0,$ for some constant $\delta>0,$ we assume that for $i  \in \mathcal{O}_k^{\tau},$ we have 
\begin{equation*}
\nu_{i}^k \geq (-\tsig_{k,i}^{-1}-x_{2k-1})^{-1/2} n^{-1/2+\delta}, 
\end{equation*}
where $\nu_{i}^k \equiv \nu_{ii}^k$ is defined in (\ref{defn_nuoij}).  


\end{assu}

\begin{prop}\label{prop_spikeweak} For fixed constant $\tau>0$ and some constant $\delta>0,$ suppose the assumptions of Theorem \ref{thm_eveout} and Assumption \ref{assum_stronger} hold true, then {Theorem \ref{thm_eveout}} holds for $k=1,2,\cdots,q, i,j \in \mathcal{O}_k^{\tau}. $
\end{prop}
\begin{proof} First of all, by (\ref{propm_equ1}),  Lemma \ref{lem_anisoout} and Theorem \ref{thm_evout}, for any $\epsilon<\min\{\tau/3, \delta\},$ there exists a high-probability event $\Xi_1 \equiv \Xi_1(\epsilon,\tau,\delta,k)$ satisfying the following conditions. \\
(i). We have 
{
\begin{equation}\label{eq_evebound1}
\mathbf{1}(\Xi_1) \norm{\Vb_o^\T\Big (z \CG_1(z)-\frac{-1}{1+m(z) \Sigma} \Big ) \Vb_o} \leq (\kappa_k+\eta)^{-1/4} n^{-1/2+\epsilon},
\end{equation}
}
for all $z$ such that $\{z \in \mathbb{C}: \Re \ z \geq a_{2k-1}+n^{-2/3+\omega},  \ |z| \leq \omega^{-1}  \}, \ \omega:=\tau/2.$ \\
(ii). For all $i$ satisfying $x_{2k-1}+n^{-1/3} \leq -\tsig_{k,i}^{-1} \leq \omega^{-1},$  we have that 
\begin{equation*}
\mathbf{1}(\Xi_1)|\mu_{k,i}-f(-\tsig_{k,i}^{-1})| \leq  (-\widetilde{\sigma}_{k,i}^{-1}-x_{2k-1})^{1/2} n^{-1/2+\epsilon}.
\end{equation*}
(iii). We have 
\begin{equation}\label{eq_rigiedge}
\mathbf{1}(\Xi_1)|\mu_{k,r_k^++1}-a_{2k-1}| \leq n^{-2/3+\epsilon}. 
\end{equation}
Till the end of the proof, we fix a realization $\CQ_1$ satisfying the above conditions and hence our discussion is purely deterministic on the high probability event $\Xi_1.$ For $i,j \in \mathcal{O}^{\tau}_k,$ we denote the radius 
\begin{equation*}
\rho_{k,i}:=\frac{ \min\big \{\nu_{i}^k, -\tsig_{k,i}^{-1}-x_{2k-1} \big\} }{2}.
\end{equation*}
We conclude from Assumption \ref{assum_stronger} that
\begin{equation}\label{rho_bound}
\rho_{k,i} \geq \frac{1}{2}(-\tsig_{k,i}^{-1}-x_{2k-1})^{-1/2}n^{-1/2+\delta}.
\end{equation}
We define the contour $\gamma_{k,i}$ as the boundary of the open disc of radius $\rho_{k,i}$ centered at $-\tsig_{k,i}^{-1}.$ We further define 
\begin{equation}\label{defn_gamma}
\Gamma_{{k,i}}:=f(\gamma_{{k,i}}).
\end{equation} 
Using a discussion similar to \cite[Lemma 5.5]{BKYY}, we conclude that  each outlier $\mu_{k,i}$ lies in $\bigcup_{i \in \mathcal{O}_k^\tau}\Gamma_{{k,i}}$ and all the other eigenvalues lie in the complement of  $\bigcup_{i \in \mathcal{O}_k^\tau}\overline{\Gamma}_{{k,i}}.$

 Since we focus on the $k$th  bulk component, we shorten our notations by setting  
\begin{equation*}
\ub_{i} \equiv \ub_{k,i}, \ \vb_{i} \equiv \vb_{k,i}, \ \Gamma_i:=\Gamma_{k,i},\ \gamma_i:=\gamma_{k,i}, d_i=d_{k,i}, \sigma_i:=\sigma_{k,i}. 
\end{equation*}
For  $i, j \in \mathcal{O}^{\tau}_k$, by spectral decomposition and Cauchy's integral formula, we have that
\begin{align}\label{vector_spec}
\langle \mathbf{u}_i, \mathbf{v}_j\rangle^2&=-\frac{1}{2 \pi \ri}\oint_{\Gamma_i} \langle \mathbf{v}_j, \CTG_1(z) \mathbf{v}_j \rangle dz \nonumber \\
& = -\frac{1}{2 \pi \ri} \oint_{\gamma_i}\langle \mathbf{v}_j, \CTG_1(f(\zeta))\mathbf{v}_j \rangle f^{\prime}(\zeta)d\zeta.
\end{align}
Together with Lemma \ref{lem_vector}, Cauchy's integral theorem and (\ref{defn_gamma}), we can further write
 \begin{equation} \label{vector_spec1}
\langle \mathbf{u}_i, \mathbf{v}_j\rangle^2=\frac{1+d_j}{d_j^2} \frac{1}{2 \pi \ri}\oint_{\Gamma_i} (D_o^{-1}+1+z\mathbf{V}_o^\T\CG_1(z)\mathbf{V}_o)^{-1}_{jj} \frac{dz}{z}.
 \end{equation}
Now we introduce the following decomposition
\begin{equation} \label{eq_resolvent}
D^{-1}_o+1+z \mathbf{V}_o^\T\CG_1(z)\mathbf{V}_o=D_o^{-1}+1-\frac{1}{1+m(z)\Lambda_o}-\Delta(z),
\end{equation}
where $\Delta(z)$ is defined as 
\begin{equation*}
 \Delta(z):=\left(-\frac{1}{1+m(z)\Lambda_o}-z\mathbf{V}_o^\T\CG_1(z) \mathbf{V}_o\right).
\end{equation*}
It is notable that $\Delta(z)$ can be well-controlled by (\ref{eq_evebound1}). By the resolvent expansion to the order of two for (\ref{eq_resolvent}) and together with (\ref{vector_spec1}),  we have that 
\begin{equation*}
\langle \mathbf{u}_i, \mathbf{v}_j \rangle^2=\frac{1+d_j}{d_j^2}(s_1+s_2+s_3),
\end{equation*}
where $s_i, i=1,2,3$ are defined as
\begingroup
\allowdisplaybreaks
\begin{align*}
& s_1:=\frac{1}{2 \pi \ri} \oint_{\Gamma_i}\frac{1}{d_j^{-1}+1-(1+m(z) \sigma_j)^{-1}} \frac{dz}{z}, \\
& s_2:=\frac{1}{2 \pi \ri} \oint_{\Gamma_i} \left(\frac{1}{d_j^{-1}+1-(1+m(z) \sigma_j)^{-1}}\right)^2 \left( \Delta(z) \right)_{jj} \frac{dz}{z},
\end{align*}
and $s_3:=\frac{1}{2 \pi \ri} \oint_{\Gamma_i} T_{jj} \frac{dz}{z},$
where $T_{jj}$ is defined as
\begin{equation*}
\left(\frac{1}{D_o^{-1}+1-(1+m(z)\Sigma_o)^{-1}}\Delta(z) \frac{1}{D_o^{-1}+1-(1+m(z)\Sigma_o)^{-1}}\Delta(z) \frac{1}{D^{-1}_o+1+z\mathbf{V}_o^\T \CG_1(z) \mathbf{V}_o} \right )_{jj}.
\end{equation*}
\endgroup

First of all, the convergent limit is characterized by $s_1.$ By the residual theorem and (\ref{key_quantity}), we have  
\begin{align*}
\frac{1+d_j}{d_j^2} s_1=\frac{1}{\tsig_j}\frac{1}{2 \pi \ri} \oint_{\gamma_i} \frac{f^{\prime}(\zeta)}{f(\zeta)} \frac{1+\zeta \sigma_j}{\zeta+\tilde{\sigma}_j^{-1}}  d \zeta 
=\delta_{ij} a_{k,i}.
\end{align*}
Next we bound $s_2$ and $s_3.$ For $s_2,$ we rewrite it as
\begin{equation*}
s_2=\frac{d_j^2}{2\tsig_j^2 \pi \ri} \oint_{\gamma_i} \frac{h_{jj}(\zeta)}{(\zeta+\tsig_j^{-1})^2}d\zeta, \ h_{jj}(\zeta):=(1+\zeta \sigma_j)^2 (\Delta(f(\zeta)))_{jj} \frac{f^{\prime}(\zeta)}{f(\zeta)}.
\end{equation*}
Since $h_{jj}(\zeta)$ is holomorphic inside the contour $\gamma_i,$ by (\ref{fderivative}) and (\ref{eq_evebound1}), we find that 
\begin{equation} \label{hbound}
|h_{jj}(\zeta)| \leq C |\zeta-x_{2k-1}|^{1/2}n^{-1/2+\epsilon}.
\end{equation} 
By Cauchy's differentiation formula, we have
\begin{equation}\label{cauchy_diff}
h^{\prime}_{jj}(\zeta)=\frac{1}{2 \pi \ri} \oint_{\mathcal{C}} \frac{h_{jj}(\xi)}{(\xi-\zeta)^2} d\xi,
\end{equation}
where $\mathcal{C}$ is the circle of radius $\frac{|\zeta-x_{2k-1}|}{2}$ centered at $\zeta.$ Hence, by (\ref{hbound}), (\ref{cauchy_diff}) and the residual theorem, we have 
\begin{equation}\label{hderivetivebound}
|h^{\prime}_{jj}(\zeta)| \leq C | \zeta-x_{2k-1}|^{-1/2}n^{-1/2+\epsilon}.
\end{equation} 
When $i=j,$ by  the residual theorem and (\ref{hderivetivebound}), we have 
\begin{equation*}
|s_2|=\left| \frac{d_i^2}{\tsig_i^2} h_{jj}^{\prime}(-\tsig_i) \right| \leq C \frac{d_i^2}{\tsig_i^2}(-\tsig_i^{-1}-x_{2k-1})^{-1/2}n^{-1/2+\epsilon}. 
\end{equation*}
When $i \neq  j,$ by Assumption \ref{assum_stronger} and residual theorem, we have  $|s_2|=0.$ Finally, we estimate $s_3.$ Here the residual calculation is not available, we need to choose precise contour for our discussion. From the definition of $s_3,$  (\ref{eq_evebound1}) and (\ref{fderivative}), by residual theorem, we find that for some constant $C>0,$
\begin{equation*}
|s_3| \leq C \oint_{\gamma_i} n^{-1+2\epsilon} |(\zeta+\tsig_j^{-1})|^{-2} \norm{(D^{-1}_o+1+f(\zeta)\mathbf{V}_o^\T \CG_1(f(\zeta)) \mathbf{V}_o)^{-1}} d|\zeta|.  
\end{equation*}
Using the resolvent  identity, we find that on $\gamma_i,$
\begin{align*}
\norm{(D^{-1}_o+1+f(\zeta)\mathbf{V}_o^\T \CG_1(f(\zeta)) \mathbf{V}_o)^{-1}}  \leq \frac{1}{ \min_t |d_t^{-1}+1+(1+\zeta \sigma_t)^{-1}|-\norm{\Delta(f(\zeta))}}.
\end{align*}
When $\zeta \in \gamma_i,$ we have that for some constant $\varsigma>0,$  
\begin{equation*}
|d_t^{-1}+1+(1+\zeta\sigma_t)^{-1}| \geq \varsigma |\zeta+\tsig_{t}^{-1}| \geq \varsigma |\zeta+\tsig_i^{-1}|=\varsigma \rho_i \geq \varsigma (-\tsig_{i}^{-1}-x_{2k-1})^{-1/2}n^{-1/2+\delta},
\end{equation*}
where in the last step we use (\ref{rho_bound}). Together with (\ref{eq_evebound1}) and the fact $\epsilon<\delta$, we find that  
\begin{equation}\label{eq_keyestimate1}
\norm{(D^{-1}_o+1+f(\zeta)\mathbf{V}_o^\T \CG_1(f(\zeta)) \mathbf{V}_o)^{-1}}  \leq C \rho^{-1}_i. 
\end{equation}
As a consequence, we find that
\begin{equation*}
|s_3| \leq C n^{-1+2\epsilon} \sup_{\zeta \in \gamma_i}  |(\zeta+\tsig_j^{-1})|^{-2}. 
\end{equation*}
Using a discussion similar to \cite[Lemma 5.6]{BKYY}, we find that 
\begin{equation*}
|\zeta+\tsig_j^{-1}| \sim \rho_i+|\tsig_j^{-1}-\tsig_i^{-1}|.
\end{equation*}
We hence conclude from the definition of $\rho_{i} \equiv \rho_{k,i}$ that 
\begin{equation*}
|s_3| \leq C n^{-1+2\epsilon} (\nu^k_{ij})^{-2},
\end{equation*}
where we use the definition of (\ref{defn_nuoij}).
This concludes our proof.

\end{proof}

Secondly, with some extra technical work, we can show that the above results still hold true without the non-overlapping condition.  We record them as the following proposition without proofs. For more details, we refer to \cite[Section 5.2]{BKYY}. 
\begin{prop}\label{prop_vectornonover}
For fixed constant $\tau>0,$ suppose the assumptions of Theorem \ref{thm_eveout} hold true, then {Theorem \ref{thm_eveout}} holds for $k=1,2,\cdots,q, i,j \in \mathcal{O}_k^{\tau}. $
\end{prop}

To complete the proof of Theorem \ref{thm_eveout}, we need to prove the results for $\tau=0.$ It will rely on some discussion from the non-outlier eigenvectors.  We will finish the proof after proving Theorem \ref{thm_evebulk}.

\subsection{Proof of Theorem \ref{thm_evebulk}}
In this section, we discuss the non-outlier eigenvectors and prove Theorem \ref{thm_evebulk}. We will also establish a result which will be used for completing the proof of Theorem \ref{thm_eveout}.  We basically follow the discussion of \cite[Proposition 6.1]{BKYY} and only sketch the proof here. 
\begin{proof}
We focus our discussion on the $k$th bulk component and use the following shorthand notations
\begin{equation*}
\mu_i \equiv \mu_{k,i}, \  \ub_{i} \equiv \ub_{k,i}, \ \theta_i \equiv \theta_{k,i}. 
\end{equation*} 
We first suppose that $j \in \mathcal{I}.$ Let $\epsilon>0$ and set $\omega:=\epsilon/2.$   By (\ref{propm_equ1}),  Lemma \ref{lem_anisoout}, Theorem \ref{thm_evout} and Lemma \ref{lem_rigidedgeuni}, there exists a high probability event $\Xi_2 \equiv \Xi_2(\epsilon, \tau,k)$ satisfying the following conditions. \\
(i). For $z \in \mathbf{S}$ defined in (\ref{full}), we have {
\begin{equation}\label{eq_bound1}
\mathbf{1}(\Xi_2) \norm{\Vb_o^\T\Big (x \CG_1(x)-\frac{-1}{1+m(x) \Sigma} \Big ) \Vb_o} \leq n^{\epsilon} \Psi(z).
\end{equation}
}
(ii). We also have  
\begin{equation*}
\mathbf{1}(\Xi_2)|\mu_{k, r^++1}-x_{2k-1}| \leq n^{-2/3+\epsilon},\ \mathbf{1}(\Xi_2)|\lambda_{k,i}-\gamma_{k,i}| \leq   i^{-1/3} n^{-2/3+\epsilon}, \  i \leq (1-\tau) n_k. 
\end{equation*}
(iii). We further have that
\begin{equation*}
\mathbf{1}(\Xi_2) |\mu_{k,i}-f(-\tsig_{k,i}^{-1})| \leq (-\tsig_{k,i}^{-1}-x_{2k-1})^{1/2} n^{-1/2+\epsilon}, 
\end{equation*}
for $\tsig_{k,i}$ satisfying 
\begin{equation*}
x_{2k-1}+n^{-1/2} \leq -\tsig_{k,i}^{-1} \leq x_{2k-1}+n^{-1/2+\tau}. 
\end{equation*}
For the following we fix a realization $\mathcal{Q}_1 \in \Xi_2$ and focus on the high probability event $\Xi_2.$ 
We set the spectral parameter $z=\mu_i+\ri \eta,$ where $\eta$ is the unique smallest solution of 
\begin{equation}\label{defn_eta1}
\operatorname{Im} m(z)=n^{-1+6\epsilon}\eta^{-1}.
\end{equation}
Hence, (\ref{eq_bound1}) reads as 
{
\begin{equation}\label{anisotropic_nonout}
\mathbf{1}(\Xi_2) \norm{\Vb_o^\T\Big (x \CG_1(x)-\frac{-1}{1+m(x) \Sigma} \Big ) \Vb_o} \leq  n^{2\epsilon} (n \eta)^{-1}. 
\end{equation}
}
Denote $\kappa \equiv \kappa_k(\mu_i),$ we find from (\ref{propm_equ1}) that 
\begin{equation}\label{eq_etaestimate}
\eta \asymp
\begin{cases}
\frac{n^{6 \epsilon}}{n \sqrt{k}+n^{2/3+2\epsilon}}, & \ \text{if} \ \mu_i \leq x_{2k-1}+n^{-2/3+4\epsilon}, \\
n^{-1/2+3\epsilon} \kappa^{1/4}, & \ \text{if} \ \mu_i \geq x_{2k-1}+n^{-2/3+4\epsilon}.
\end{cases}
\end{equation}
Now we start our estimation.  From the spectral decomposition, we conclude that 
\begin{equation}\label{spec_out}
\langle \mathbf{u}_i, \mathbf{v}_{j} \rangle^2 \leq \eta \mathbf{v}_j^\T \operatorname{Im} \CTG_1(z) \mathbf{v}_j.
\end{equation} 
Since $j \in \mathcal{I},$ by Lemma \ref{lem_vector} and the resolvent identity  for (\ref{eq_resolvent}), we have 
{\small
\begin{align*} \label{nonoutlarge}
&z \langle \mathbf{v}_j, \CTG_1(z) \mathbf{v}_j \rangle   \\
& = d_j^{-1}-(1+d_j)d_j^{-2} \left((d_j^{-1}+1-(1+m(z)\sigma_j)^{-1})^{-1}+(d_i^{-1}+1-(1+m(z)\sigma_j)^{-1})^{-2} \left( \Delta(z) \right)_{ii}  \right.  \nonumber \\
&  \left. +  \left((\mathbf{D}_o^{-1}+1-(1+m(z)\Lambda_o)^{-1})^{-1}\Delta(z) (\mathbf{D}_o^{-1}+1-(1+m(z)\Lambda_o)^{-1})^{-1}\Delta(z) \frac{1}{\mathbf{D}_o^{-1}+1+z\mathbf{V}_o^\T \CG_1(z) \mathbf{V}_o} \right )_{ii} \right) . \nonumber
\end{align*} 
}
First of all, we estimate the error item containing two $\Delta(z)'s. $ Using a discussion similar to (\ref{eq_keyestimate1}),  by (\ref{anisotropic_nonout}) and (\ref{eq_etaestimate}), we have that 
\begin{equation*}
\left| \left| \frac{1}{\mathbf{D}^{-1}_o+1+z\mathbf{V}_o^\T \CG_1(z) \mathbf{V}_o} \right| \right| \leq \frac{2 }{\operatorname{Im}m(z)}=2n^{1-6 \epsilon} \eta.
\end{equation*}
As a consequence, we have 
\begin{equation*}
z \langle \mathbf{v}_j, \CTG_1(z) \mathbf{v}_j \rangle= -(1+m(z) \tsig_j)^{-1}+O(n^{2 \epsilon} |1+m(z) \tsig_j|^{-2} (n \eta)^{-1}).  
\end{equation*}
%
%
Hence, together with (\ref{spec_out}), for some constant $C>0,$ we have 
\begin{align} \label{bounform}
\langle \mathbf{u}_i,  \mathbf{v}_j \rangle^2  & \leq - \eta \text{Im} (z^{-1}(1+m(z) \tsig_j)^{-1})+C n^{2 \epsilon} (n |z|)^{-1}  |1+m(z) \tsig_j|^{-2} \nonumber \\
& =   -\eta^2 |z|^{-2} \text{Re} ((1+m(z) \tsig_j)^{-1})-\eta |z|^{-2} \mu_i \text{Im}  ((1+m(z) \tsig_j)^{-1}) \nonumber \\
&+ C n^{2 \epsilon} (n|z|)^{-1}|1+m(z) \tsig_j|^{-2}. 
\end{align}

The estimate will reply on the following lemma, whose proof is similar to \cite[eq. (6.10)]{BKYY} and we omit the details here. 
\begin{lem}\label{lem_denominatorbound}
For any fixed $\delta \in [0, 1/3-\epsilon),$ we have the lower bound 
\begin{equation*}
|1+\tsig_j m(z)| \geq \varsigma (n^{-2 \delta}|-\tsig_j^{-1}-x_{2k-1}|+\operatorname{Im}\ m(z)),
\end{equation*}
for $\mu_i \in [0, f(x_{2k-1}+n^{-1/3+\delta+\epsilon})]$ and some constant $\varsigma>0.$
\end{lem}

We now estimate the items in (\ref{bounform}) one by one. Since $z$ is bounded by Lemma \ref{lem_rigi_equ},  for some constant $C>0,$ we have that 
\begin{align*}
-\eta^2 |z|^{-2} \text{Re} ((1+m(z) \tsig_j)^{-1})  \leq  \frac{C \eta^2}{|1+m(z) \tsig_j|} \leq C \frac{ \eta^2}{\text{Im} \ m(z)}=C\eta^3 n^{1-6\epsilon} ,
\end{align*}
where in the second inequality we use Lemma \ref{lem_denominatorbound}. It is easy to see that $\eta \leq n^{-2/3+4 \epsilon+\delta}$ using (\ref{fderivative}) and (\ref{eq_etaestimate}), we hence conclude that 
\begin{equation*}
-\eta^2 |z|^{-2} \text{Re} ((1+m(z) \tsig_j)^{-1})  \leq C n^{-1+6 \epsilon+3 \delta}. 
\end{equation*}
Next, we find that 
\begin{align*}
-\eta |z|^{-2} \mu_i \text{Im}  ((1+m(z) \tsig_j)^{-1})  \leq C \frac{ \eta \text{Im} \ m(z)}{|1+m(z) \tsig_j|^2} \leq \frac{C n^{6 \epsilon}}{n |1+m(z) \tsig_j|^2}. 
\end{align*}
Finally, we can estimate 
\begin{equation*}
C n^{2 \epsilon} (n|z|)^{-1}|1+m(z) \tsig_j|^{-2} \leq C \frac{n^{2 \epsilon}}{n|1+m(z) \tsig_j|^2 }.
\end{equation*}
Putting all these estimates together, we conclude that 
\begin{equation*}
\langle \mathbf{u}_i,  \mathbf{v}_j \rangle^2  \leq C (n^{-1+6 \epsilon+3 \delta}+n^{-1+6 \epsilon}|1+m(z) \tsig_j|^{-2}). 
\end{equation*}
Finally, we will use Lemma \ref{lem_denominatorbound} to estimate the right-hand side of the above equation.  For the $k$th bulk component when $i \geq r_k^++1,$ by (\ref{propm_equ1}), Condition (ii) of $\Xi_2$  and (\ref{eq_etaestimate}), for some constant $\varsigma>0,$ we have
\begin{equation*}
\text{Im} \ m(z) \geq \varsigma \sqrt{\theta_i}. 
\end{equation*}
By choosing $\delta=0$ in Lemma \ref{lem_denominatorbound}, we can complete the proof for $j \in \mathcal{I}$.  The general case follows from a limiting argument. We denote $\widehat{\mathcal{I}}:=\mathcal{I} \cup \{j\}$ and consider 
\begin{equation}\label{eq_modifiedmodel}
\widehat{\widetilde{\Sigma}}= \Sigma(1+\widehat{\Vb}_o \widehat{D}_{o} \widehat{\Vb}^\T_{o}), \ \widehat{\Vb}_{o}:=(\vb_k)_{k \in \widehat{\mathcal{I}}}, \ \widehat{D}_{o}:=\text{diag}\{\widehat{d}_k\}_{k \in \widehat{\mathcal{I}}},
\end{equation}
where $\widehat{d}_k:=d_k$ for $k \in \mathcal{I}$ and $\widehat{d}_j \in (0,1/2).$ Since $|\widehat{\mathcal{R}}| \leq r+1$ and $\widehat{D}_o$ is invertible, we can apply the above results for $j \in \mathcal{I}$ to the modified model (\ref{eq_modifiedmodel}). Now by taking the limit $\widehat{d}_j \rightarrow 0$ we can conclude the proof in the general case. 

\end{proof}

\begin{rem}
Note that when $i \leq r_k^+$ satisfies that $-\tsig_{k,i}^{-1} \leq x_{2k-1}+n^{-1/3+\tau},$ we need to choose $\delta=\tau$ for Lemma \ref{lem_denominatorbound} since in which case 
\begin{equation*}
\text{Im} \ m(z) \geq \varsigma \sqrt{\eta}.
\end{equation*} 
This leads to the estimate 
\begin{equation}\label{eq_morestiate}
\langle \mathbf{u}_i,  \mathbf{v}_j \rangle^2  \prec C n^{-1+3 \tau}(|-\tsig_j^{-1}-x_{2k-1}|^2+\theta_i)^{-1}. 
\end{equation}
The above estimate will be used to complete the proof of Theorem \ref{thm_eveout}. 
\end{rem}

\subsection{Proof of Theorem \ref{thm_eveout}} In this section, we prove Theorem \ref{thm_eveout} by allowing $\tau=0$ in $\mathcal{O}_k^{\tau},$ whereas the case $\tau>0$ has been proved in Section \ref{sec_spikedvector}.  
\begin{proof}
Fix $\epsilon>0.$ Note that it is easy to check by contradiction that there exists some $s \in [r_k]$ satisfying the gap condition: for all $t$ such that $-\tsig_t^{-1}>x_{2k-1}+sn^{-1/3+\epsilon}$  we have $-\tsig_t^{-1} \geq x_{2k-1}+(s+1)n^{-1/3+\epsilon}.$ For such $s,$ we can decompose $\mathcal{O}_k=\mathcal{O}_{k0} \cup \mathcal{O}_{k1}$ such that $-\tsig_t^{-1} \leq x_{2k-1}+sn^{-1/3+\epsilon}$ for $t \in \mathcal{O}_{k0} \cap \mathcal{O}$ and $-\tsig_t^{-1} \geq x_{2k-1}+(s+1)n^{-1/3+\epsilon}$    
for $t \in \mathcal{O}_{k1} \cap \mathcal{O}.$ It suffices to consider the case when $i \in \mathcal{O}_{k0}$ since the case $\mathcal{O}_{k1}$ follows from Proposition \ref{prop_vectornonover}. We first consider the case when $j \in \mathcal{O}_{k0}.$ Since  $-\tsig_i^{-1}-x_{2k-1} \leq s n^{-1/3+\epsilon},$ we find from (\ref{fderivative}) that 
\begin{equation*}
a_{k,i} \sim (-\tsig_i^{-1}-x_{2k-1}) \leq s n^{-1/3+\epsilon}, \ \nu_{ij}^k \leq C n^{-1/3+\epsilon},
\end{equation*} 
for some constant $C>0.$ Setting $\tau=\epsilon$ in (\ref{eq_morestiate}), we find that (\ref{eveout_eq}) holds with an extra $n^{3 \epsilon}$ times the right-hand side of (\ref{eveout_eq}). The proof $j \in \mathcal{O}_{k1}$ is similar.  Finally, the proof of (\ref{eveout_non}) follows from a similar discussion and the assumption  (\ref{defnregularityequation1}). We omit the details here. 
\end{proof}

\noindent \section*{Acknowledgements}
I would like to thank Jeremy Quastel and Balint Virag for fruitful discussions and valuable suggestions.  I also want to thank Weihao Kong for the discussion of some statistical applications and motivations. Finally, the author is grateful to two anonymous referees and the editor for their important comments and suggestions, which have significantly improved  the paper.

\begin{appendix}

\end{appendix}

\end{document}